\documentclass[11pt]{amsart}

\title[Extending automorphisms of the genus-$2$ surface over the $3$-sphere]
{Extending automorphisms of the genus-$2$ surface over the $3$-sphere}

\author{Kenta Funayoshi}
\address{
Department of Mathematics \newline
\indent Hiroshima University, 1-3-1 Kagamiyama, Higashi-Hiroshima, 739-8526, Japan}
\email{kfunayoshi@yahoo.co.jp}

\author{Yuya Koda}
\thanks{The second author is supported in part by JSPS KAKENHI Grant
 Numbers 15H03620, 17K05254, 17H06463, and JST CREST Grant Number JPMJCR17J4.}
\address{
Department of Mathematics \newline
\indent Hiroshima University, 1-3-1 Kagamiyama, Higashi-Hiroshima, 739-8526, Japan}
\email{ykoda@hiroshima-u.ac.jp}

\usepackage{amsthm}
\usepackage{latexsym}
\usepackage[dvips]{graphicx}
\usepackage[dvips]{psfrag}
\usepackage[dvips]{color}
\usepackage{xypic}
\usepackage[abs]{overpic}
\usepackage{caption}
\usepackage[all]{xy}

\theoremstyle{plain}
\newtheorem*{theorem*}{Theorem}
\newtheorem*{lemma*} {Lemma}
\newtheorem*{corollary*} {Corollary}
\newtheorem*{proposition*}{Proposition}
\newtheorem*{conjecture*}{Conjecture}
\newtheorem{theorem}{Theorem}[section]
\newtheorem{lemma}[theorem]{Lemma}
\newtheorem{corollary}[theorem]{Corollary}

\newtheorem{claim}{Claim}
\newtheorem*{subclaim}{Subclaim}

\theoremstyle{remark}

\newtheorem*{definition}{Definition}

\newtheorem*{remark}{Remark}

\newtheorem*{problem}{Problem}

\theoremstyle{definition}

\newtheoremstyle{citing}
  {}
  {}
  {\itshape}
  {}
  {\bfseries}
  {.}
  {.5em}
  {\thmnote{#3}}

\theoremstyle{citing}
\newtheorem*{citingtheorem}{} 

\newcommand{\Integer}{\mathbb{Z}}
\newcommand{\Real}{\mathbb{R}}

\newcommand{\MCG}{\mathrm{MCG}}
\newcommand{\Homeo}{\mathrm{Homeo}}

\newcommand{\rel}{\mathrm{rel}}

\newcommand{\Nbd}{\operatorname{Nbd}}
\newcommand{\Cl}{\operatorname{Cl}}
\newcommand{\Int}{\operatorname{Int}}



\begin{document}

\maketitle

\begin{abstract}
An automorphism $f$ of a closed orientable surface $\Sigma$ 
is said to be extendable over the 3-sphere $S^3$ 
if $f$ extends to an automorphism of the 
pair $(S^3, \Sigma)$ with respect to some embedding 
$\Sigma \hookrightarrow S^3$. 
We prove that if an automorphism of a genus-2 surface $\Sigma$ 
is extendable over $S^3$, then 
$f$ extends to an automorphism of the 
pair $(S^3, \Sigma)$ with respect to 
an embedding $\Sigma \hookrightarrow S^3$ 
such that $\Sigma$ bounds genus-2 handlebodies on both sides.  
The classification of essential annuli in the exterior of genus-2 
handlebodies embedded in $S^3$ due to Ozawa and the second author 
plays a key role. 
\end{abstract}

\vspace{1em}

\begin{small}
\hspace{2em}  \textbf{2010 Mathematics Subject Classification}: 
57M60; 57S25

%

\hspace{2em} 
\end{small}

\section*{Introduction}

Let $\Sigma_g$ be the closed orientable surface of genus $g$. 
An automorphism $f$ of $\Sigma_g$ is said to be {\it null-cobordant} if 
$f$ extends to an automorphism of a compact orientable 3-manifold $M$ 
with respect to an identification 
$\Sigma_g = \partial M$. 
In Bonahon \cite{Bon83}, it is proved that if $f$ is an orientation-preserving, null-cobordant, periodic automorphism of $\Sigma_g$, 
then the above $M$ can be chosen to be a handlebody. 
It is also showed in the same paper that the same consequence holds for 
an arbitrary irreducible automorphism of $\Sigma_2$.  
These results give partial answers to the problem: 
Given an automorphism $f$ of $\Sigma_g$ that extends to an automorphism 
of a certain 3-manifold $M$ with $\partial M = \Sigma_g$, find 
the ``simplest'' 3-manifold among such $M$ for $f$. 
In fact, a handlebody of genus $g$ is definitely the ``simplest'' 3-manifold with the boundary 
$\Sigma_g$. 

An automorphism $f$ of $\Sigma_g$  is said to be {\it extendable over $S^3$} if 
$f$ extends to an automorphism of the pair $(S^3, \Sigma_g)$ 
with respect to an embedding $\Sigma_g \hookrightarrow S^3$.
In particular, if we can choose the above embedding $\Sigma_g \hookrightarrow S^3$ 
to be {\it standard} 
(i.e. an embedding such that $\Sigma_g$ bounds handlebodies on both sides), 
we say that $f$ is {\it standardly extendable over $S^3$}. 
An automorphism of $\Sigma_g$ extendable over $S^3$ is clearly null-cobordant, but the converse is false. 
Indeed, the Dehn twist along a non-separating simple closed curve on $\Sigma_g$ 
is null-cobordant, but not extendable over $S^3$ (see Section \ref{sec:Automorphisms extendable over S3}).
As an analogy of the above-mentioned problem for null-cobordant automorphisms, 
the following problem naturally arises:  
\begin{problem}
Given an automorphism $f$ of $\Sigma_g$ that 
extends to an automorphism of the pair $(S^3, \Sigma_g)$ 
with respect to an embedding $\Sigma_g \hookrightarrow S^3$, 
find the ``simplest'' such an embedding $\Sigma_g \hookrightarrow S^3$ 
for $f$. 
\end{problem}
The best candidate for the ``simplest'' embedding  $\Sigma_g \hookrightarrow S^3$ 
is definitely a standard one. There is a series of studies by 
Guo-Wang-Wang \cite{WWZ16}, Guo-Wang-Wang-Zhang \cite{GWWZ} and 
Wang-Wang-Zhang-Zimmermann \cite{WWZZ13, WWZZ15, WWZZ18} 
considering finite subgroups of the automorphism group $\Homeo(\Sigma_g)$ 
of $\Sigma_g$ 
that extend to subgroups of $\Homeo (S^3, \Sigma_g )$ 
with respect to some $\Sigma_g \hookrightarrow S^3$.
A direct consequence of their results is that 
a periodic automorphism of $\Sigma_2$ that is extendable over $S^3$ 
is actually standardly extendable over $S^3$.  
In the present paper, we show that we can remove the periodicity condition here. 
In fact, we prove the following: 
\begin{citingtheorem}[Theorem \ref{theo-main}]
Let $f$ be an automorphism of a closed orientable surface of genus two. 
If $f$ is extendable over $S^3$, then $f$ is standardly extendable over $S^3$. 
\end{citingtheorem}
It is easily seen that the same fact is valid for a closed surface of genus less than two as well, 
see Section \ref{sec:Automorphisms extendable over S3}. 
The case of higher genera remains open. 
The classification of essential annuli in the exteriors of 
genus-two handlebodies embedded in $S^3$ 
(see Section \ref{sec:Essential annuli in the exterior of a handlebody of genus two in S3}) given by \cite{KO} 
plays a key role in our proof of Theorem \ref{theo-main}.  
\vspace{1em}

Throughout the paper, we will work in the piecewise linear category. 
Any surfaces in a 3-manifold are always assumed to be properly embedded, and their intersection is transverse and minimal up to isotopy. 
For convenience, we will not distinguish 
surfaces, compression bodies, e.t.c.  
from their isotopy classes in their notation.
$\Nbd(Y)$ will denote a regular neighborhood of $X$, $\Cl(X)$ the closure of $X$,  and 
$\Int (X)$ the interior of $X$ for a subspace $X$ of a space, where the ambient space will always be clear from the context. 
The number of components of $X$ is denoted by $\# X$. 
Let $M$ be a 3-manifold, and let $L \subset M$ be a submanifold, or a graph. 
When $L$ is 1 or 2-dimensional, we write 
$E(L) = \Cl ( M \setminus \Nbd (L))$. 
When $L$ is of 3-dimension, we write 
$E(L) = \Cl ( M \setminus L)$.

\section{Preliminaries}
\label{sec:Preliminaries}

Let $L_1, L_2 , \ldots, L_n, R$ be possibly empty subspaces 
of a compact orientable 3-manifold $M$. 
We will denote by $\Homeo (M, L_1, L_2 , \ldots, L_n ~\rel~ R )$ 
the group of 
automorphisms of $M$ which {map} $L_i$ onto 
$L_i$ for any $i=1 , 2 , \ldots , n$ and 
which {are} {the} identity on $R$. 
The {\it mapping class group}, denoted by 
$\MCG (M , L_1 , L_2 , \ldots, L_n ~\rel~ R )$, 
is defined to be the group of isotopy classes of 
elements of $\Homeo (M , L_1 , L_2 , \ldots, L_n ~\rel~ R )$. 
When $R = \emptyset$, we will drop $\rel~R$. 
The ``plus" subscripts, for instance in 
$\Homeo_+ (M, L_1, L_2 , \ldots, L_n  ~\rel~ R )$ and 
$\MCG_+ (M, L_1, L_2 , \ldots, L_n ~\rel~ R )$, 
indicate the subgroups of 
$\Homeo (M, L_1, L_2 , \ldots, L_n  ~\rel~ R )$ and 
$\MCG (M, L_1, L_2 , \ldots, L_n ~\rel~ R )$, 
respectively, consisting of 
orientation-preserving automorphisms 
(or their classes) of $M$.

\subsection{Handlebodies}
\label{subsec:Handlebodies}

Let $V$ be a handlebody. 
A simple closed curve $c$ on $\partial V$ is 
said to be {\it primitive} if there exists a disk $D$ properly embedded in $V$ 
such that the two loops $c$ and $\partial D$ intersect transversely in a single point. 
Suppose that $V$ is embedded in $S^3$ so that the exterior $W := E(V)$ is also 
a handlebody. 
Then a disk $D$ properly embedded in $V$ is said to be {\it primitive} 
if $\partial D$ is primitive in $W$. 

\begin{lemma}\label{2.2}
Let $V$ be a handlebody of genus two, and $E$ be a separating essential disk in $V$. 
Let $c$ be a simple closed curve on $\partial V$ with $\partial E \cap c = \emptyset$ 
and $[c] \neq 1 \in \pi_1(V)$. 
Then there exists a unique non-separating disk $D$ in $V$ with $\partial D \cap c  =\emptyset$. 
Further, this disk $D$ is disjoint from $E$. 
\end{lemma}
\begin{proof}
The disk $E$ cuts $V$ into two solid tori $X_1$ and $X_2$. 
Without loss of generality, we can assume that 
$c \subset X_2$. 
Then a meridian disk $D$ in $X_1$ with $\partial D \subset \partial V$ is disjoint from 
$c$, and furthermore, disjoint from $E$. 

We show the uniqueness of $D$. 
Assume that there exists a non-separating disk $D'$ in $V$ 
that is disjoint from $c$ and not isotopic to $D$. 
If $D \cap D' = \emptyset$, then $D'$ is a meridian disk of a solid torus $V$ cut off by $D$. 
Since $[c] \neq 1 \in \pi_1(V)$, $D'$ intersects $c$. 
This is a contradiction. 
Suppose that $D \cap D^\prime \neq \emptyset$.
Then any outermost subdisk of $D'$ cut off by $D \cap D'$ must intersect $c$, 
a contradiction again. 
\end{proof}

\begin{corollary}
\label{primitive}
Let $V$ be a handlebody of genus two, and $c$ be a primitive curve on $\partial V$. 
Then there exists a unique non-separating disk $D$ in $V$ with 
$\partial D \cap c  = \emptyset$. 
\end{corollary}

\begin{proof}
Since $c$ is primitive, there exists a disk $D_0$ properly embedded in $V$ 
such that the two loops $c$ and $\partial D_0$ intersect transversely in a single point. 
Then $E := \Cl (\partial \Nbd(c \cup D_0) \setminus \partial V )$ is 
a separating disk in $V$ disjoint from $c$. 
The assertion thus follows from Lemma \ref{2.2}. 
\end{proof}

\begin{lemma}
\label{case2}
Let $V$ be a handlebody of genus two, and $E$ be an essential separating disk in $V$. 
Let $X_1$ and $X_2$ be the solid tori $V$ cut off by $E$, and $c_i$ $(i=1,2)$ be an essential simple closed curve on 
$\partial V \cap \partial X_i$ satisfying 
$[c_i] \neq 1 \in \pi _1(V)$. 
Then $E$ is the unique essential separating disk in $V$ disjoint from $c_1 \cup c_2$. 
\end{lemma}
\begin{proof}
Let $E'$ be an essential separating disk in $V$ disjoint from $c_1 \cup c_2$. 
Let $D_i$ be a meridian disk of $X_i$ satisfying $\partial D \subset \partial V$. 
Let $S$ be the $4$-holed sphere $\partial V$ cut off by $\partial D_1 \cup \partial D_2$.
Denote by $d_i^+$ and $d_i^-$ ($i=1,2$) the boundary circles of $S$ coming from
$\partial D_i$.  
Then $c_i \cap S$ consists of parallel arcs connecting $d_i^+$ and $d_i^-$. 
It suffices to show that 
$E ^ \prime \cap (D_1 \cup D_2) = \emptyset$. 
Assume for a contradiction that 
$E^\prime \cap (D_1 \cup D_2) \neq \emptyset$. 
Let $\delta$ be an outermost subdisk of $E'$ cut off by $E' \cap (D_1 \cup D_2)$. 
Then $\alpha := \partial \delta \cap S$ is a separating arc in $S$. 
Without loss of generality, we can assume that 
the end points of $\alpha$ line in $d_2^+$. 
See Figure \ref{fig:outermost}.  
\begin{center}
\begin{overpic}[width=5cm,clip]{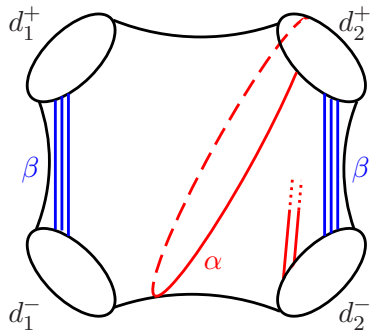}
  \linethickness{3pt}
  \put(74,25){\color{red} $\alpha$}
  \put(5,60){\color{blue} $\beta$}
  \put(130,60){\color{blue} $\beta$}
  \put(0,115){$d_1^+$}
  \put(125,115){$d_2^+$}
  \put(0,5){$d_1^-$}
  \put(125,5){$d_2^-$}
\end{overpic}
\captionof{figure}{The arc $\alpha$ and the arcs $\beta$.}
\label{fig:outermost}
\end{center}
Put $\beta := \partial E' \cap S$. 
We note that $\beta$ consists of essential arcs in $S$, and $\alpha \subset \beta$. 
Then it is easily seen that 
$\#( \beta \cap d_2^-) + 2 \leq \# ( \beta \cap d_2^+) $, 
which is a contradiction.  
\end{proof}

A pair $(V, k)$ of a handlebody $V$ and the union $k = \cup_{i=1}^n k_i$ of mutually disjoint, mutually non-parallel, simple closed curves $k_1, k_2, \ldots , k_n$ on $\partial V$ is called a 
{\it relative handlebody}. 
A relative handlebody $(V, k)$ is said to be {\it boundary-irreducible} 
if $\partial V \setminus k$ is incompressible in $V$. 
A surface $S$ in $V$ is said to be {\it essential} in $(V, k)$ 
if $S$ is an incompressible surface disjoint from $k$, and 
for every disk $D \subset V$ such that 
$D \cap S = \Cl ( \partial D - \partial V)$ 
is connected and $D \cap k = \emptyset$, 
there is a disk $D' \subset S$ with $\Cl (\partial D' - \partial S) = D \cap S$ and 
$D' \cap k = \emptyset$.

Denote by $\mathcal{A} (V, k)$ the subgroup of the mapping class
group $\MCG (V, k)$ generated by all twists along essential annuli in $(V, k)$. 
\begin{lemma}[Johannson \cite{Jo1}, Remark 8.9]
\label{rel-handlebody}
Let $(V, k)$ be a boundary-irreducible relative handlebody. 
Then $\MCG (V, k) / \mathcal{A} (V, k)$ is a finite group. 
\end{lemma}

\subsection{Compression bodies}

Let $V$ be a handlebody. 
A (possibly empty, possibly disconnected) subgraph of a spine of $V$ is called 
a {\it subspine} of $V$ if it does not contain a contractible component. 
A {\it compression body} $W$ is the exterior $E(\Gamma)$ 
of a subspine $\Gamma$ of a handlebody $V$. 
The component $\partial_+ W = \partial V$ is called 
the {\it exterior boundary} of
$W$, and $\partial_- W = \partial W \setminus \partial_+ W = \partial \Nbd (\Gamma)$ is called the {\it interior boundary} of $W$. 
We remark that the interior boundary is incompressible in $W$, see Bonahon \cite{Bon83}.

A compression body $W$ with $\partial_+ W$ a closed orientable surface of genus two 
is one of the following (see Figure \ref{fig:ccb}): 
(i) $\partial_- W = \emptyset$;
(ii) $\partial_- W$ is a torus; 
(iii) $\partial_- W$ consists of two tori;  
(iv) $\partial_- W$ is a closed orientable surface of genus two. 
\begin{center}
\begin{overpic}[width=15cm,clip]{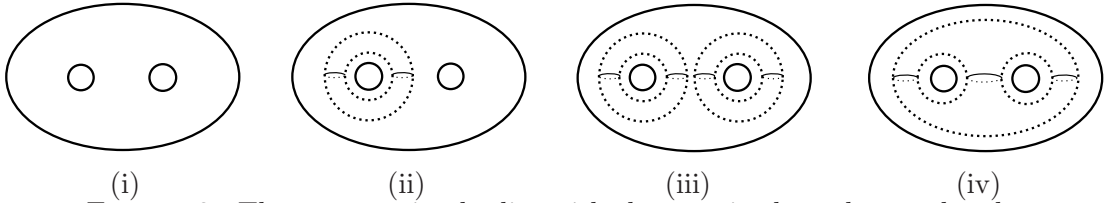}
  \linethickness{3pt}
  \put(45,-10){(i)}
  \put(150,-10){(ii)}
  \put(255,-10){(iii)}
  \put(365,-10){(iv)}
\end{overpic}
\captionof{figure}{The compression bodies with the exterior boundary a closed surface of genus two.}
\label{fig:ccb}
\end{center}

\begin{theorem}[Bonahon \cite{Bon83}, Theorem 2.1]
\label{thm:Bonahon}
For a compact, orientable, irreducible $3$-manifold $M$ with $\partial M$ connected,  
there exists a unique compression body $W$ in $M$ such that 
$\partial_+ W = \partial M$, and that the closure of $M \setminus W$ is boundary irreducible. 
\end{theorem}
The compression body $W$ in Theorem \ref{thm:Bonahon} is called 
the {\it characteristic compression body} of $M$. 

\begin{remark}
In \cite{Bon83}, Theorem $\ref{thm:Bonahon}$ is stated under a more general setting. 
\end{remark}

\subsection{W-decompositions}
\label{subsec:W-decompositions}

Let $M$ be a boundary-irreducible Haken 3-manifold. 
An essential annulus or torus $S$ in $M$ is said to be {\it canonical} 
if any other essential annulus or torus in $M$ can be isotoped to be disjoint from $S$. 
A maximal system $\{ S_1, S_2, \ldots, S_n\}$ of pairwise disjoint, pairwise non-parallel 
canonical surfaces in $M$ is called a {\it W-system}. 

\begin{lemma}[Neumann-Swarup \cite{NS}]
\label{W-system}
The W-system of a boundary-irreducible Haken 3-manifold is unique up to isotopy. 
\end{lemma}
The result of cutting $M$ off by a W-system is called a {\it W-decomposition} of $M$.   

Let $M_1$, $M_2, \ldots, M_m$ be the result of performing 
the W-decomposition of $M$. 
We set $\partial_0 M_i := \partial M_i \cap \partial M$ and 
$\partial_1 M_i := \partial M_i \setminus \partial_0 M_i$ ($i = 1, 2, \ldots, m$). 
We say that $(M_i, \partial_0 M_i)$ (or simply $M_i$) is {\it simple} if any essential annulus or torus 
$(S, \partial S) \subset (M_i, \partial_0 M_i)$ is parallel to $\partial_1M_i$. 

\begin{lemma}[Neumann-Swarup \cite{NS}] 
\label{prop-simple}
If $M_i$ is not simple, then $M_i$ is either a Seifert fibered space or an $I$-bundle. 
\end{lemma}

\section{Essential annuli in the exterior of a genus-two handlebody in $S^3$}
\label{sec:Essential annuli in the exterior of a handlebody of genus two in S3}

Essential annuli in the exterior of a handlebody of genus two 
embedded in $S^3$ play an important role in our paper. 
We first briefly review the classification of the essential annuli 
in the exterior of a handlebody of genus two embedded in $S^3$ obtained in \cite{KO}. 

\begin{description}

\item[{\rm Type 1}]
Let $\Gamma$ be a handcuff-graph embedded in $S^3$.  
Let $S$ be a sphere in $S^3$ that intersects $\Gamma$ 
in exactly one edge of $\Gamma$ twice and transversely. 
Set $V := \Nbd(\Gamma)$. 
We call $A := S \setminus \Int (V)$ 
a {\it Type $1$ annulus} for $V \subset S^3$. 
See Figure \ref{fig:type1}. 
We note that if $V \subset S^3$ admits an essential annulus of Type 1, then there exists an annulus component $B$ of $\partial V$ cut off by 
$\partial A$, and $A \cup B$ forms an essential torus in $E(V)$. 
\begin{center}
\begin{overpic}[width=12cm,clip]{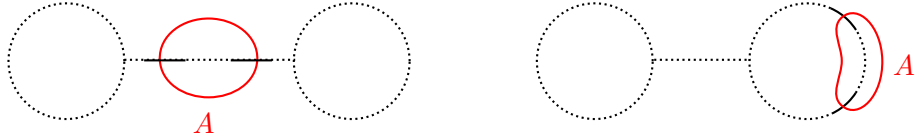}
  \linethickness{3pt}
  \put(75,0){\color{red} $A$}
  \put(340,22){\color{red} $A$}
\end{overpic}
\captionof{figure}{Type 1 annuli.}
\label{fig:type1}
\end{center}

\item[{\rm Type 2}]
Let $\Gamma$ be a handcuff-graph embedded in $S^3$.  
Assume that one of the two loops of 
$\Gamma$ is a trivial knot bounding a disk 
$D$ such that 
$\Int (D)$ intersects $\Gamma$ in an edge $e$ once and transversely.  
Set $V := \Nbd(\Gamma)$ and $A := D \cap E(V)$. 
We call $A$ a {\it Type $2$ annulus} for $V \subset S^3$. 
When $e$ is a loop, we say that $A$ is of {\it Type $2$-$1$}, and 
when $e$ is the cut-edge, we say that  $A$ is of {\it Type $2$-$2$}. 
See Figure \ref{fig:type2}. 
\begin{center}
\begin{overpic}[width=8cm,clip]{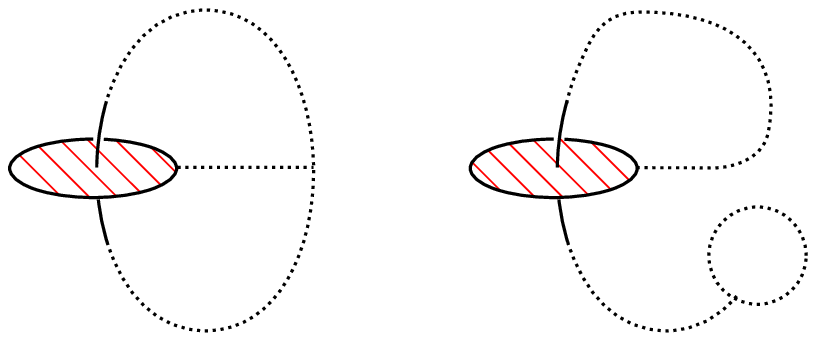}
  \linethickness{3pt}
  \put(45,0){(i)}
  \put(170,0){(ii)}
  \put(15,61){\color{red} $A$}
  \put(140,61){\color{red} $A$}
\end{overpic}
\captionof{figure}{(i) A Type 2-1 annulus, (i) A Type 2-2 annulus.}
\label{fig:type2}
\end{center}

\item[{\rm Type 3}]
Let $X$ be a solid torus embedded in $S^3$. 
Let $A$ be an annulus properly embedded in $E(X)$ so that 
$\partial A \cap \partial X$ consists of 
parallel non-trivial simple closed curves on 
$\partial X$. 
\begin{itemize}
\item
Let $\tau$ be a properly embedded trivial simple arc  
in $X$ with $\partial \tau \cap \partial A = \emptyset$. 
Set $V := X \setminus \Int ( \Nbd(\tau) )$. 
Then we call $A$ a {\it Type $3$-$1$ annulus} for $V \subset S^3$ provided that, 
if $\partial A$ bounds an essential disk in $X$, then any meridian disk of $X$ has non-empty 
intersection with $\tau$. 
We note that if $V \subset S^3$ admits an essential annulus of Type 3-1, 
$E(V)$ is not boundary-irreducible. 
\item
Suppose that $\partial A$ does not bound an essential disk in $X$. 
Let $\tau$ be a properly embedded simple arc 
in $E(X)$ with $\tau \cap A = \emptyset$. 
Set $V := X \cup \Nbd(\tau)$. 
Then we call $A$ a {\it Type $3$-$2$ annulus} for $V \subset S^3$. 
\end{itemize}

Let $X_1$, $X_2$ be two disjoint solid tori embedded in $S^3$. 
Assume that there exists an annulus $A$ properly embedded in 
$E(X_1 \sqcup X_2)$ such that 
$A \cap \partial X_i$ $(i=1,2)$ is a simple closed curve on 
$\partial X_i$ that does not bound a disk in $X_i$. 
Let $\tau \subset E(X_1 \sqcup X_2) \setminus A$ be 
a proper arc connecting 
$\partial X_1$ and $\partial X_2$. 
Set $V := X_1 \cup X_2 \cup \Nbd (\tau)$. 
Then we call $A$ a {\it Type $3$-$3$ annulus} for $V \subset S^3$. 
An annulus $A$ in the exterior of a genus two handlebody $V \subset S^3$ 
is called a {\it Type $3$ annulus} if it is one of 
Types $3$-$1$, $3$-$2$ and $3$-$3$ annuli. 
Figure \ref{fig:type3} shows 
schematic pictures of Type 3 annuli. 
\begin{center}
\begin{overpic}[width=13cm,clip]{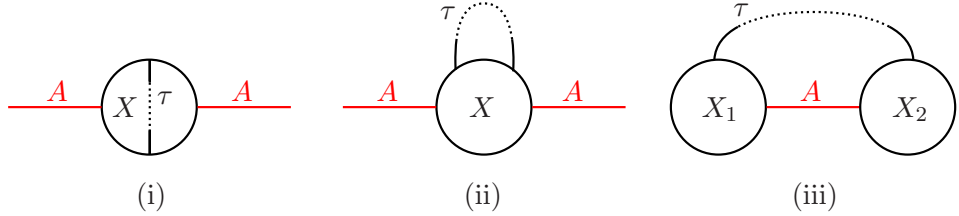}
  \linethickness{3pt}
  \put(54,0){(i)}
  \put(178,0){(ii)}
  \put(302,0){(iii)}
  \put(45,33){$X$}
  \put(180,33){$X$}
  \put(268,33){$X_1$}
  \put(340,33){$X_2$}
  \put(62,38){$\tau$}
  \put(169,70){$\tau$}
  \put(280,71){$\tau$}
  \put(20,39){\color{red} $A$}
  \put(90,39){\color{red} $A$}
  \put(145,39){\color{red} $A$}
  \put(215,39){\color{red} $A$}
  \put(305,39){\color{red} $A$}
\end{overpic}
\captionof{figure}{(i) A Type 3-1 annulus, (ii) A Type 3-2 annulus, (iii) A Type 3-3 annulus.}
\label{fig:type3}
\end{center}

\item[{\rm Type 4}]
Let $K$ be a knot whose exterior $E(K)$ contains 
an essential two-holed torus $P$ such that 
\begin{itemize}
\item 
$P$ cuts $E(K)$ into two handlebodies $V$ and $W$ of genus two; and  
\item
$\partial P$ consists of parallel non-integral slopes on $\partial \Nbd(K)$. 
\end{itemize}
Such a knot is called an {\it Eudave-Mu\~noz knot}.  
See Eudave-Mu\~noz \cite{EM2} for more details of such knots. 
Set $A := \partial \Nbd(K) \setminus 
\Int (\partial \Nbd(K) \cap \partial V)$. 

\begin{itemize}
\item
We call $A$ a {\it Type $4$-$1$ annulus} 
for $V \subset S^3$. 
\item
Let $U \subset S^3$ be a knot or a two component link 
contained in $W$ so that 
$W \setminus \Int (\Nbd(U))$ 
is a compression body. 
Let $i : E(U) \to S^3$ be a re-embedding 
such that $E(i(E(U)))$ is not a solid torus or two solid tori. 
Then we call $i(A)$ 
a {\it Type $4$-$2$ annulus} 
for $i (V) \subset S^3$. 
We note that if $V \subset S^3$ admits an essential annulus of Type 4-2, 
$E(V)$ contains an essential torus. 
\end{itemize}
An annulus $A$ in the exterior of a genus two handlebody $V \subset S^3$ 
is called a {\it Type $4$ annulus} if it is one of 
Types $4$-$1$ and $4$-$2$. 
Figure \ref{fig:type4} depicts schematic picture of 
an essential annulus of Type 4. 
\begin{center}
\begin{overpic}[width=6cm,clip]{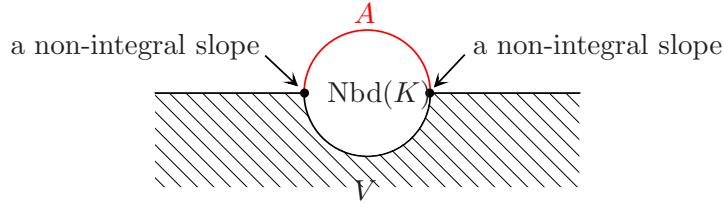}
  \linethickness{3pt}
  \put(70,37){$\Nbd(K)$}
  \put(80,0){$V$}
  \put(-50,55){a non-integral slope}
  \put(125,55){a non-integral slope}
  \put(80,67){\color{red} $A$}
\end{overpic}
\captionof{figure}{A Type 4 annulus.}
\label{fig:type4}
\end{center}
\end{description}

\begin{theorem}[\cite{KO}]
\label{thm:classification of essential annuli}
Let $V$ be a handlebody of genus two embedded in $S^3$. 
Then each essential annulus in the exterior of $V$ 
belongs to exactly one of the four Types 
listed above. 
\end{theorem}

The following is straightforward. 

\begin{lemma}
\label{lem:non-existence of Types 1, 3-1 and 4-2}
Let $V$ be a handlebody of genus two embedded in $S^3$ 
so that $E(V)$ is boundary-irreducible and atoroidal.  
Then $E(V)$ does not admit Types $1$, $3$-$1$ or $4$-$2$ annuli. 
\end{lemma}

Existence of an essential annulus $A \subset E(V)$ sometimes 
obstructs the existence of essential annuli of some other types 
disjoint from $A$. 
In the following part of this section, we will be concerned with 
this feature. 
 
\begin{lemma}
\label{lem:uniqueness of Type 2-1}
Let $V$ be a handlebody of genus two embedded in $S^3$ 
such that $E(V)$ is boundary-irreducible and atoroidal.  
If $E(V)$ admits disjoint essential annuli $A, A'$ of Type $2$-$1$, 
then $A'$ is isotopic to $A$ in $E(V)$. 
\end{lemma}
\begin{proof}
Let $A, A' \subset E(V)$ be disjoint essential annuli of Type $2$-$1$. 
Set $\partial A = a_1 \sqcup a_2$ and $\partial A' = a_1' \sqcup a_2'$, where 
$a_1$ and $a_1'$ bound disks in $V$, and 
$a_2$ and $a_2'$ are primitive curves on $\partial V$. 
Since $a_2$ is primitive, we have $a_1 = a_1'$ by Corollary \ref{primitive}. 
By slightly pushing the interior of the annulus $\hat{A} = A \cup A'$ into $\Int ( E(V) )$, 
$\hat{A}$ becomes an incompressible annulus in $E(V)$ with 
$\partial \hat{A} = a_2 \cup a_2'$ (see Figure \ref{fig:pushing_annulus}). 
We will see that $a_2$ is isotopic to $a_2^\prime$ on $\partial V$. 
If $\hat{A}$ is parallel to $\partial E(V)$, then clearly 
$a_2$ is isotopic to $a_2^\prime$. 
Suppose that $\hat{A}$ is not parallel to $\partial E(V)$.  
By Lemma 1.4 of \cite{KO}, $\hat{A}$ is essential in $E(V)$. 
Recall that each of $a_2$ and $a_2^\prime$ is primitive (and so non-separating) in $\partial V$.  
Since $E(V)$ does not admit none of Types 1, 3-1 and 4-2 annuli by 
Lemma \ref{lem:non-existence of Types 1, 3-1 and 4-2}, 
such an annulus $\hat{A}$ is one of Types 3-2, Type 3-3 and 4-1 annuli. 
If $\hat{A}$ one of Types 3-2 and 4-1 annuli, 
the boundary circles of $\hat{A}$ is parallel on $\partial V$ by the definitions. 
Suppose that $\hat{A}$ is a Type 3-3 annulus. 
By the definition of a Type 3-3 annulus, 
there exists a separating essential disk $E$ in $V$\\ 
\begin{center}
\begin{overpic}[width=9cm,clip]{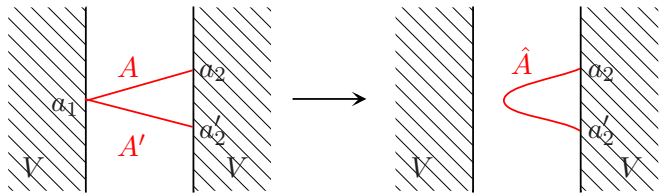}
  \linethickness{3pt}
  \put(10,10){$V$}
  \put(87,10){$V$}
  \put(157,10){$V$}
  \put(236,10){$V$}
  \put(21,35){$a_1$}
  \put(77,48){$a_2$}
  \put(77,25){$a_2'$}
  \put(224,48){$a_2$}
  \put(224,25){$a_2'$}
  \put(46,48){\color{red} $A$}
  \put(46,18){\color{red} $A'$}
  \put(195,49){\color{red} $\hat{A}$}
\end{overpic}
\captionof{figure}{The annulus $\hat{A}$.}
\label{fig:pushing_annulus}
\end{center}
that cuts $V$ off into two solid tori 
$X_1$ and $X_2$ with $a_2 \subset \partial X_1$ and $a_2' \subset \partial X_2$. 
Note that $a_2$ ($a_2'$, respectively) does not bound a disk in $X_1$ ($X_2$, respectively). 
By Lemma \ref{2.2}, the meridian disk $D$ of $X_2$ is the unique non-separating disk in $V$ 
disjoint from $a_2$. 
Thus we have $\partial D = a_1$. 
This implies that $a_1 \cap a_2' \neq \emptyset$, a contradiction. 
Therefore, in any case we have $a_1 = a_1'$ and $a_2 = a_2'$. 
By moving $A'$ by isotopy so that $a_2 = a_2^\prime$ and 
slightly pushing it into $\Int (E(V))$, 
$\hat{A}$ becomes a torus in $E(V)$. 
Since $E(V)$ is atoroidal and irreducible, $\hat{A}$ bounds a solid torus $X$ in $E(V)$. 
Since $a_1$ bounds a disk in $V$, 
$a_1$ intersects the meridian of $X$ once and transversely. 
In consequence $A'$ can be isotoped to $A$ through $X$. 
\end{proof}

\begin{lemma}
\label{lem:Types 2-1 and 3-2}
Let $V$ be a handlebody of genus two embedded in $S^3$.  
If there exists an essential annulus $A$ of Type $2$-$1$ in $E(V)$, 
then $E(V)$ does not admit an essential annulus of Type $3$-$2$ disjoint from $A$. 
\end{lemma}
\begin{proof}
Let $A \subset E(V)$ be an essential annulus $A$ of Type $2$-$1$. 
Suppose that there exists an essential Type 3-2 annulus $A' \subset E(V)$. 
We will show that $A \cap A' \neq \emptyset$. 
Set $\partial A = a_1 \sqcup a_2$ and $\partial A' = a_1' \sqcup a_2'$, 
where $a_1$ bounds a disk in $V$ and 
$a_2$ is a primitive curve on $\partial V$. 
We note that by the definition of a Type 3-2 annulus, 
$a_1'$ and $a_2'$ are parallel non-separating curves on $\partial V$. 
Suppose for a contradiction that $A \cap A' = \emptyset$. 
Let $D$ be a disk in $V$ with $\partial D = a_1$. 
By cutting $V$ along $D$, we get an unknotted solid torus $X$ in $S^3$. 
Then $a_1$, and so $a_i^\prime$ ($i = 1, 2$), is the preferred longitude of 
$\partial X$. 
We note that $A'$ is a boundary-parallel annulus in $E(X)$. 
Now $V$ is obtained by attaching a 1-handle $\Nbd(D)$ to $X$, however, 
anyhow we attach such a 1-handle to $X$, 
$A^\prime$ becomes boundary-compressible in $E(V)$. 
This contradicts the assumption that $A^\prime$ is essential in $E(V)$. 
\end{proof}

\begin{lemma}
\label{lem:A is an essential annulus in E(X)}
Let $V$ be a handlebody of genus two embedded in $S^3$, 
and let $A \subset E(V)$ be an essential annulus of Type $3$-$2$. 
Let $X$ be a solid torus for $A$ in the definition of a Type $3$-$2$ annulus. 
Then exactly one of the following $\mathrm{(i)}$ and $\mathrm{(ii)}$ holds: 
\begin{enumerate}
\renewcommand{\labelenumi}{(\roman{enumi})}
\item
$X$ is a regular neighborhood of a torus knot or a cable knot, 
and $A$ is the cabling annulus for $X$. 
\item
$A$ is an incompressible annulus in $E(X)$ parallel to an annulus $A'$ on $\partial X$. 
There exists and arc $\tau$ contained in the parallelism
region between $A$ and $A'$, with endpoints on $A'$, such that 
$V = X \cup \Nbd (\tau)$. 
\end{enumerate}
\end{lemma}
\begin{proof}
Set $\partial A = a_1 \sqcup a_2$. 
We first show that $A$ is incompressible in $E(X)$. 
Suppose that there exists a compressing disk $D$ for $A$. 
We note that $\partial D$ is the core of $A$. 
If $E(X)$ is not the solid torus, 
$a_i$ ($i=1,2$) is a non-essential simple closed curve on $\partial X$. 
In this case each component of $\partial A$ bounds a disk in $X$, so does in $V$. 
This contradicts the assumption that $A$ is a Type 3-2 annulus. 
Thus $E(X)$ is the solid torus. 
If at least one of $a_1$ and $a_2$ are non-essential on $\partial X$, 
we get a contradiction by the same reason as above. 
Thus both $a_1$ and $a_2$ are the preferred longitudes of $X$. 
Then $A$ is parallel to an annulus $A'$ on $\partial X$ by 
Lemma 3.1 of Kobayashi \cite{Kobayashi}.
Let $Y$ be the parallelism region between $A$ and $A'$. 
There exist a boundary-compressing disk $D'$ for $A$ in $E(X)$. 
We note that $D$ lies in $E(X) - Y$ while $D'$ lies in $Y$. 
Thus $A$ remains incompressible or boundary-compressible even in $E(V)$, 
which is a contradiction. 
The annulus $A$ is thus incompressible in $E(X)$. 

If $A$ is essential in $E(X)$, 
the the property $\mathrm{(i)}$ follows from Lemma 15.26 of Burde-Zieschang \cite{Burde-Zieschang}.
Suppose $A$ is not essential in $E(X)$. 
Then again by Kobayashi \cite{Kobayashi} in the case where $E(X)$ is the solid torus, and 
by  Lemma 15.18 of Burde-Zieschang \cite{Burde-Zieschang} in the other case, 
$A$ is parallel to $\partial X$. Hence the the property $\mathrm{(ii)}$ follows. 
\end{proof}

Lemma \ref{lem:A is an essential annulus in E(X)} indicates that  
the boundary of a Type 3-2 annulus in the exterior of a handlebody $V$ 
in $S^3$ consists of parallel primitive curves on $\partial V$. 

\begin{lemma}
\label{lem:uniqueness of X for Type 3-2}
Let $V$ be a handlebody of genus two embedded in $S^3$, 
and let $A \subset E(V)$ be an essential annulus of Type $3$-$2$. 
Then, for $A$, the solid torus $X$ in the definition of a Type $3$-$2$ annulus is 
uniquely determined. 
\end{lemma}
\begin{proof}
Set $\partial A = a_1 \sqcup a_2$. 
Recall that $a_1$ and $a_2$ are parallel primitive curves on $\partial V$. 
By Corollary \ref{primitive}, there exists a unique non-separating disk 
disjoint from $a_1 \cup a_2$. 
This implies that the solid torus $X$ in the definition of a 
Type 3-2 annulus is uniquely determined. 
Indeed, $X$ is obtained by cutting $V$ off by that the non-separating disk.  
\end{proof}

Let $V$ be a handlebody of genus two embedded in $S^3$, 
and let $A \subset E(V)$ be an essential annulus of Type 4-1. 
The two circles $\partial A$ are parallel on $\partial V$, thus 
cobound an annulus $A'$ on $\partial V$. 
We note that the union $A \cup A^\prime$ is the boundary of 
a regular neighborhood $\Nbd(K)$ of an 
Eudave-Mu\~noz knot $K$. 
The annulus $A$ cuts $E(V)$ off into a solid torus $\Nbd(K)$ and 
a handlebody $W$ of genus two. 
The core $k_{A'}$ of the annulus $A'$ is an essential simple closed curve on $V$. 
Set $P := \Cl (\partial V \setminus A')$. 
$P$ is a two-holed torus. (Recall Figure \ref{fig:type4}).

\begin{lemma}
\label{Case3-lemm}
The relative handlebody $(V, k_{A^\prime})$ admits no non-separating essential annuli. 
\end{lemma}
\begin{proof}
We suppose for a contradiction that 
there exists a non-separating essential annulus 
$B$ in $(V, k_{A'})$. 
Set $\partial A$ = $a_1 \sqcup a_2$ and $\partial B$ = $b_1 \sqcup b_2$. 
Let $Y$ be the result of Dehn surgery on $K$ along the slope $k_{A'}$. 
That is, $Y$ is obtained by attaching a solid torus $X$ to $E(K)$ along their 
boundaries so that $k_{A'}$ bounds a disk in $X$.  
Let $D_i$ ($i=1,2$) be the disks in $X$ bounded by $a_i$. 
Then $T := P \cup (D_1 \cup D_2)$ is an essential torus in $Y$. 
By Proposition 5.4 of Eudave-Mu\~noz \cite{EM} 
each component $Y_i$ ($i=1,2$) of $Y$ cut off by $T$ is a Seifert fibered space with 
the base space a disk and two exceptional fibers. 
Without loss of generality we can assume that 
$V \subset Y_1$ and $W \subset Y_2$. 
Since $B$ is non-separating in $V$, so is in $Y_1$.
This contradicts Theorem {\rm V\hspace{-1pt}I}.34 of Jaco \cite{Jaco}, which 
asserts that there is no non-separating essential annuli in $Y_1$. 
\end{proof}

\section{Automorphisms extendable over $S^3$}
\label{sec:Automorphisms extendable over S3}
It is well known that every closed orientable $3$-manifold $M$ 
can be decomposed into two handlebodies $V$ and $W$ of the same genus $g$ for some $g \geq 0$.
That is, $V \cup W = M$ and $V \cap W = \partial V = \partial W = \Sigma_g$, a 
closed orientable surface of genus $g$.
Such a decomposition is called a {\it Heegaard splitting} for the manifold $M$.
The surface $\Sigma_g$ is called the {\it Heegaard surface} of the splitting, and 
$g$ is called the {\it genus} of the splitting.
By Waldhausen \cite{W} there exists a unique Heegaard splitting of a given genus for $S^3$ up to isotopy.
We say that an embedding $\iota: \Sigma_g \hookrightarrow S^3$ is {\it standard} 
if the image $\iota (\Sigma_g)$ is a Heegaard surface in $S^3$. 

\begin{definition}
\label{defi-standard}
Let $f$ be an automorphism of a closed orientable surface $\Sigma_g$. 
Then $f$ is said to be {\it extendable over $S^3$} if there exist an embedding 
$\iota: \Sigma_g \hookrightarrow S^3$ and an automorphism 
$\hat{f} : (S^3, \iota(\Sigma_g)) \rightarrow (S^3, \iota(\Sigma_g))$ 
satisfying 
$\hat{f} \circ \iota = \iota \circ f$. 
In particular, if we can choose the above $\iota$ to be standard, we say that 
$f$ is {\it standardly extendable over $S^3$}. 
\end{definition}
It is easy to see that the Dehn twist $T_{c}$ along a separating simple closed curve $c$ 
on $\Sigma_g$ is standardly extendable over $S^3$. 
Indeed, there exists a standard embedding $\iota : \Sigma_g \hookrightarrow S^3$ 
such that $\iota (c)$ bounds disks on both sides. 
On the other hand, the Dehn twist $T_{c'}$ along a non-separating simple closed curve $c'$ 
on $\Sigma_g$ is not extendable over $S^3$. 
This is because if $T_{c'}$ is extendable over $S^3$, there should exist a non-separating 
2-sphere in $S^3$ by Theorem \ref{McC} below, which is a contradiction. 

\begin{theorem}[McCullough \cite{McC}] \label{McC}
Let $M$ be a compact orientable $3$-manifold, and $c_1$, $c_2 , \ldots, c_n$ be 
mutually disjoint simple closed curves on $\partial M$. 
If $($powers of$)$ Dehn twists along $c_1$, $c_2 , \ldots, c_n$ extend to 
an automorphism of $M$, then for each $i \in \{1, 2, \ldots, n\}$, 
either $c_i$ bounds a disk in $M$, or for some
$j \neq i$, $c_i$ and $c_j$ cobound an incompressible annulus in $M$. 
\end{theorem}

In \cite{GWWZ}, Guo-Wang-Wang-Zhang determined all periodic automorphisms 
of the closed surface $\Sigma_2$ of genus two that are extendable over $S^3$.  
They showed that among the twenty-one 
(conjugacy classes of) periodic maps of $\Sigma_2$, 
exactly thirteen maps are extendable over $S^3$. 
They described those maps explicitly, and as a direct consequence, 
we have the following:   
\begin{theorem}[Guo-Wang-Wang-Zhang \cite{GWWZ}]
\label{finite}
Let $f$ be a periodic automorphism of a closed surface of genus two. 
If $f$ is extendable over $S^3$, then $f$ is standardly extendable over $S^3$. 
\end{theorem}

The following is our main theorem: 

\begin{theorem}
\label{theo-main}
Let $f$ be an automorphism of a closed surface of genus two. 
If $f$ is extendable over $S^3$, then $f$ is standardly extendable over $S^3$. 
\end{theorem}

We remark that in \cite{WWZZ15}, Wang-Wang-Zhang-Zimmermann 
found some finite, non-cyclic subgroups of $\Homeo (\Sigma_g)$ that 
extend to subgroups of $\Homeo (S^3, \Sigma_g )$ 
with respect to some $\Sigma_g \hookrightarrow S^3$, 
but do not extend to those of $\Homeo (S^3, \Sigma_g )$ 
with respect to any standard embedding $\Sigma_g \hookrightarrow S^3$.

The extendability (standardly extendability, respectively) over $S^3$ 
is preserved under conjugation. 
Indeed, suppose that an automorphism $f$ of $\Sigma_g$ 
extends to an automorphism $\hat{f}$ of the pair $(S^3, \iota(\Sigma_g))$ 
with respect to an embedding (a standard embedding, respectively) 
$\iota : \Sigma_g \hookrightarrow S^3$.  
Then for any automorphism $h$ of $\Sigma_g$, 
$f^\prime := h^{-1} \circ f \circ h$ extends to 
$\hat{f}$ 
with respect to the embedding (a standard embedding, respectively) $\iota \circ h$. 

The extendability and standardly extendability, respectively over $S^3$ are preserved 
under isotopy as well. 
In the remaining of the paper, we will not distinguish 
homeomorphisms 
from their isotopy classes in their notation unless otherwise mentioned.
Here we recall the cyclic Nielsen realization theorem: 
\begin{theorem}[Nielsen \cite{Nielsen}]\label{Nielsen}
Every element $f \in \MCG ( \Sigma_g )$ $(g \geq 2)$ of finite order $k$ 
has a representative $\phi \in \Homeo ( \Sigma_g )$ of order $k$.
\end{theorem}
By this theorem, we have the mapping class version of Theorem \ref{finite} as follows: 
\begin{lemma}
\label{lem:finite order}
Let $f \in \MCG ( \Sigma_2 )$ be an element of finite order.  
If $f$ is extendable over $S^3$, then $f$ is standardly extendable over $S^3$. 
\end{lemma}

\begin{remark}
The same consequence of Theorem \ref{theo-main} 
holds for a closed surface $\Sigma_g$ with $g \leq 1$ as well. 
The case of $\Sigma_0$ is trivial. 
In fact, every automorphism of $\Sigma_0$ is standardly extendable 
over $S^3$ by the Sch\"onflies Theorem. 
Suppose that an automorphism $f$ of $\Sigma_1$ 
extends to an automorphism $\hat{f}$ of the pair $(S^3, \iota(\Sigma_1))$ 
with respect to an embedding 
$\iota : \Sigma_1 \hookrightarrow S^3$.  
Note that at least one of the two components of $S^3$ cut off by 
$\iota (\Sigma_0)$ is a solid torus $V$. 
Then $E(V)$ is the exterior of a knot. 
If $E(V)$ is also a solid torus, there is nothing to prove. 
Suppose that $E(V)$ is not a solid torus. 
We note that in this case we have $\hat{f}(V) = V$. 
Let $\mu$ and $\lambda$ be the meridian and preferred longitude of $V \subset S^3$. 
By the uniqueness of $\mu$ and $\lambda$, the automorphism $\hat{f}$ preserves $\mu$ and $\lambda$. 
Thus, it is straightforward that $f$ is standardly extendable 
over $S^3$. 
\end{remark}

\section{Proof of Theorem \ref{theo-main}.}
\label{sec:Proof of Theorem}

Let $f$ be an automorphism of a closed surface $\Sigma$ of genus two. 
Suppose that $f$ extends to an automorphism $\hat{f}$ of the pair 
$(S^3, \iota (\Sigma))$ 
with respect to an embedding $\iota : \Sigma \hookrightarrow S^3$. 
The surface $\iota (\Sigma)$ separates $S^3$ into two compact 3-manifolds $M_1$ and $M_2$, 
i.e. $S^3 = M_1 \cup_{\iota(\Sigma)} M_2$. 
Suppose for a contradiction that $f$ is not standardly extendable over $S^3$.

\begin{claim}
\label{claim1}
After changing $\iota$ $($and so $\hat{f})$ if necessary, 
we can assume that $M_1$ is a handlebody and 
$M_2$ is boundary-irreducible. 
\end{claim}

\begin{proof}[Proof of Claim $\ref{claim1}$]
By Theorem \ref{thm:Bonahon} there exist 
the characteristic compression bodies $W_1$ and $W_2$ of 
$M_1$ and $M_2$, respectively. 
Then each of $W_1$ and $W_2$ is homeomorphic to 
one of the four compression bodies shown in Figure \ref{fig:ccb}. 

Since $\iota ( \Sigma )$ is compressible in $S^3$, both $\partial_- W_1$ and $\partial_- W_2$ cannot be a closed surface of genus two. 
Both $\partial_- W_1$ and $\partial_- W_2$ cannot be empty as well, otherwise $f$ is standardly extendable over $S^3$, 
which contradict our assumption. 
If $\partial_- W_1 = \emptyset$ ($\partial_- W_2 = \emptyset$, respectively) and 
$\partial_- W_2$ ($\partial_- W_1$, respectively) is a closed surface of genus two, 
there is nothing to prove. 
Thus we only need to consider the case where at least one of 
$\partial_- W_1$ and $\partial_- W_2$ is a torus or the disjoint union of two tori. 

Suppose for example that each of $\partial_- W_1$ and $\partial_- W_2$ is a torus. 
Set 
$T_i := \partial_- W_i = \partial W_i \setminus \iota (\Sigma)$ ($i=1,2$).
Then we can write  
$S^3 =  ( W_1 \cup M_2 ) \cup_{T_1} \Cl (S^3 \setminus (W_1 \cup M_2))$. 
By the uniqueness (Theorem \ref{thm:Bonahon}) 
of the characteristic compression bodies $W_1$ and $W_2$ of 
$M_1$ and $M_2$, respectively, 
we may assume that $\hat{f}$ preserves $W_1 \cup W_2$ setwise. 
Further, by the definition of a characteristic compression body, 
$E (W_1 \cup M_2) = \Cl (M_1 \setminus W_1)$ is 
boundary-irreducible. 
In particular, $E (W_1 \cup M_2)$ is not a solid torus. 
Since any torus in $S^3$ bounds a solid torus, $W_1 \cup M_2$ is a solid torus. 
By the same reason, $W_2 \cup M_1$ is a solid torus. 
Let $\lambda_1$ and $\lambda_2$ be the preferred longitudes of $W_1 \cup M_2$ 
and $W_2 \cup M_1$, respectively. 
We attach two solid tori $X_1$ and $X_2$ to $W_1 \cup W_2$ by 
homeomorphisms $\phi_i : \partial X_i \to T_i$ ($i=1,2$) that 
send the meridian of $\partial X_i$ to $\lambda_i$. 
Then the resulting manifold $X_1 \cup_{\phi_1} (W_1 \cup W_2) \cup_{\phi_2} X_2$ is again 
the 3-sphere $S^3$ and $\iota (\Sigma)$ is now a Heegaard surface. 
(This operation corresponds to a re-embedding of $\Sigma$ into $S^3$).  
See Figure \ref{fig:re-embedding}. 
\begin{center}
\begin{overpic}[width=12cm,clip]{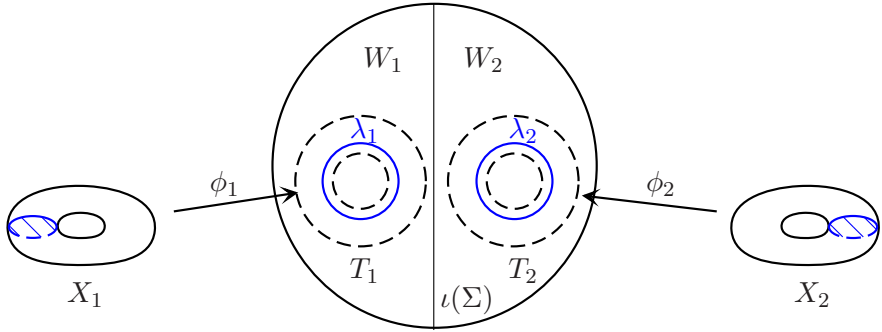}
  \linethickness{3pt}
  \put(28,16){$X_1$}
  \put(303,16){$X_2$}
  \put(82,58){$\phi_1$}
  \put(247,58){$\phi_2$}
  \put(135,78){\color{blue} $\lambda_1$}
  \put(195,78){\color{blue} $\lambda_2$}
  \put(140,105){$W_1$}
  \put(178,105){$W_2$}
  \put(135,25){$T_1$}
  \put(195,25){$T_2$}
  \put(169,14){$\iota(\Sigma)$}

\end{overpic}
\captionof{figure}{The manifold $X_1 \cup_{\phi_1} (W_1 \cup W_2) \cup_{\phi_2} X_2 \cong S^3$.}
\label{fig:re-embedding}
\end{center}
By the uniqueness of the preferred longitude, 
$\hat{f}|_{W_1 \cup W_2}$ preserves $\lambda_1 \sqcup \lambda_2$ setwise. 
Therefore $\hat{f}|_{W_1 \cup W_2}$ extends to an automorphism of 
$X_1 \cup_{\phi_1} (W_1 \cup W_2) \cup_{\phi_2} X_2$. 
This implies that $f$ is standardly extendable over $S^3$, which is a contradiction. 

In a similar way, we can show that 
if $\partial_- W_1$ (say) is empty and $\partial_- W_2$ is a torus; 
$\partial_- W_1$ (say) is empty and $\partial_- W_2$ consists of two tori; 
$\partial_- W_1$ (say) is a torus and $\partial_- W_2$ consists of two tori; or 
each of $\partial_- W_1$ and $\partial_- W_2$ consists of two tori, 
we can show that $f$ is standardly extendable over $S^3$, a contradiction. 

When $\partial_- W_1$ (say) is a torus and $\partial_- W_2$ is a closed surface of gens two; 
or $\partial_- W_1$ (say) consists of two tori and $\partial_- W_2$ is a closed surface of gens two, a similar argument as above concludes the assertion. 
\end{proof}

By virtue of Claim \ref{claim1}, in the following we assume that 
$S^3 = V \cup_{\iota(\Sigma)} M$, 
where $V$ is a handlebody of genus two, and $M$ is boundary-irreducible. 

\begin{claim}
\label{claim2}
After changing $\iota$ $($and so $\hat{f})$ if necessary, 
we can assume that $M$ is atoroidal as well. 
\end{claim}
\begin{proof}[Proof of Claim $\ref{claim2}$]
According to the torus decomposition theorem 
by Jaco-Shalen \cite{JS} and Johannson \cite{Jo2}, 
there exists a unique minimal family $ \mathcal{T} = \{ T_1, T_2, \ldots, T_m \}$ of mutually disjoint, 
mutually non-parallel, essential tori such that 
each component of the manifolds obtained by cutting $M$ along $\cup_{i=1}^m T_i$ is 
either a Seifert fibered space or a simple manifold. 
By the uniqueness of $\mathcal{T}$, we may assume that 
$\hat{f}$ preserves $\cup_{i=1}^m T_i$ setwise. 
By the same argument as in Claim \ref{claim1}, 
we can eliminate $ \mathcal{T} = \{ T_1, T_2, \ldots, T_m \}$ 
by re-embedding $\Sigma$ into $S^3$. 
Note that if some tori are nested, we can eliminate several tori at once by applying 
the operation in Claim \ref{claim1} to the outermost one. 
\end{proof}

Due to Claim \ref{claim2}, in the remaining of the proof we assume that 
$S^3 = V \cup_{\iota(\Sigma)} M$, 
where $V$ is a handlebody of genus two, and $M$ is boundary-irreducible and atoroidal. 

Suppose that $M$ does not contain essential annuli. 
Then by Johannson \cite{Joh79}, $\MCG(M)$ is a finite group. 
Note that this fact can also be checked as follows. 
By Thurston's hyperbolization theorem \cite{Kap01} and 
equivariant torus theorems \cite{Hol91}, $M$ admits 
a hyperbolic structure with totally geodesic boundary.  
Thus, due to the well-known fact that 
the isometry group of a compact hyperbolic manifold is finite, 
see Kojima \cite{Koj88}, $\MCG(M)$ is finite. 
Consequently, the order of $f \in \MCG(\Sigma)$ is finite. 
Now by Lemma \ref{lem:finite order}, $f$ is standardly extandable over $S^3$, 
a contradiction. 

Suppose that $M$ contains an essential annulus. 
By Lemma \ref{W-system} 
there exists a unique W-system $\mathcal{W}$ of $M$. 
This set $\mathcal{W}$ is not empty. 
To show this, suppose for a contradiction that 
$\mathcal{W} = \emptyset$. 
Then the result of performing the W-decomposition of $M$ is $M$ itself. 
Since $M$ contains an essential annulus, clearly $M$ is not simple. 
Hence by Lemma \ref{prop-simple}, $M$ is either a Seifert fibered space or an $I$-bundle.
Since $\partial M$ is a closed surface of genus two, $M$ is not a Seifert fibered space. 
Thus $M$ is an $I$-bundle over a compact surface $S$. 
If $\partial S \neq \emptyset$, then $M$ is a handlebody, which is a contradiction. 
If $S$ is an orientable closed surface, then $\partial M$ is not connected, 
a contradiction as well. 
If $S$ is a non-orientable closed surface, then 
$H_1(M)$ contains a non-trivial torsion. 
This contradicts $H_1(M) \cong \Integer \oplus \Integer$. 

\vspace{1em}

Since $M$ is atoroidal, $\mathcal{W}$ consists of essential annuli. 
We will use the classification of the essential annuli in $W$ introduced in 
Section \ref{sec:Essential annuli in the exterior of a handlebody of genus two in S3}. 
Since $M$ is boundary-irreducible and atoroidal, $M$ admits 
none of Types 1, 3-1 and 4-2 essential annuli. 
In the following, we will prove the non-existence of essential 
annuli of the other Types. 

\begin{claim}
\label{claim:Type 2-2}
The W-system $\mathcal{W} $ does not contain Type $2$-$2$ annuli.  
\end{claim}
\begin{proof}[Proof of Claim $\ref{claim:Type 2-2}$]
Suppose that $\mathcal{W}$ contains a Type 2-2 annulus $A$. 
Let $a_0$ be the component of $\partial A$ that is the boundary of 
an essential disk $E$ in $V$, and let $a_1$ be the other component. 
By cutting $V$ along $E$, we have two solid tori $X_1$ and $X_2$. 
Without loss of generality we can assume that $a_1 \subset \partial X_1$. 
By the definition of a Type 2-2 annulus, $a_1$ is the preferred longitude of $X_1$. 
Let $\mu_i$ ($i=1,2$) be the meridian of $X_i$, and 
$\lambda_2$ be the preferred longitude of $X_2$ 
(see Figure \ref{fig:type2-2_claim} (i)). 
By the uniqueness of $\mathcal{W}$ (Lemma \ref{W-system}), 
it holds $\hat{f}^n (A) = A$ for some natural number $n$. 
Using $2n$ instead of $n$ if necessary, 
we can assume that 
$\hat{f}^n|_V$ is orientation-preserving. 
Since $a_0$ is separating on $\partial V$ whereas $a_1$ is non-separating, 
$\hat{f}^n$ does not exchange $a_0$ and $a_1$. 
Thus $\hat{f}^n$ preserves each of $\mu_1$, $\mu_2$ and $\lambda_2$. 
It follows from Cho \cite{Cho} that $\hat{f}^n |_V$ is 
a composition of $\alpha$, $\beta_1$, $\beta_2$ defined as follows. 
The map $\alpha$ comes from the hyperelliptic involution of $\partial V$. 
The maps $\beta_1$ and $\beta_2$ are the half-twists of handles shown in 
Figure \ref{fig:type2-2_claim} (ii). 
We note that $(\beta_1 \circ \beta_2)|_{\partial V} = T_{a_0}$, 
where $T_{a_0}$ is the Dehn twist along $a_0$. 
Then we have ${\beta_1}^2|_{\partial V} = {\beta_2}^2|_{\partial V} =T_{a_0}$ and  $\beta_1 \circ {\beta_2}^{-1} = \alpha$. 
Since $\alpha$, $\beta_1$ and $\beta_2$ are mutually commutative, 
we can write $\hat{f}^n|_V$ in the form: 
\[ \hat{f}^n|_V = \alpha^\epsilon \circ {\beta_1}^{n_{1}} \circ {\beta_2}^{n_{2}},~
(\epsilon \in \{ 0,1\}, ~ n_i \in \Integer).\]
Since 
$(\hat{f}^n|_V)^2 = {\beta_1}^{2n_{1}} \circ {\beta_2}^{2n_{2}}$, 
it holds $(\hat{f}^n|_{\partial V})^2 = {T_{a_0}}^{n_{1} + n_{2}}$. 
Suppose that $n_{1} + n_{2} \neq 0$.
Sine $(\hat{f}^n|_{\partial V})^2$ is extendable over $S^3$, 
$a_0$ bounds a disk in $M$ by Theorem \ref{McC}. 
This contradicts the boundary-irreducibility of $M$. 
Hence $n_1 + n_2 = 0$, and thus we have 
$\hat{f}^n|_V = \alpha \circ {\beta_1}^{n_1} \circ {\beta_2}^{-n_1}$. 
Recalling that $\beta_1 \circ {\beta_2}^{-1} = \alpha$, 
we see that $\hat{f}^n|_V$ is either the identity or $\alpha$. 
By Lemma \ref{lem:finite order}, $f$ is standardly extendable over $S^3$, 
which is a contradiction. 
\end{proof}
\begin{center}
\begin{overpic}[width=14cm,clip]{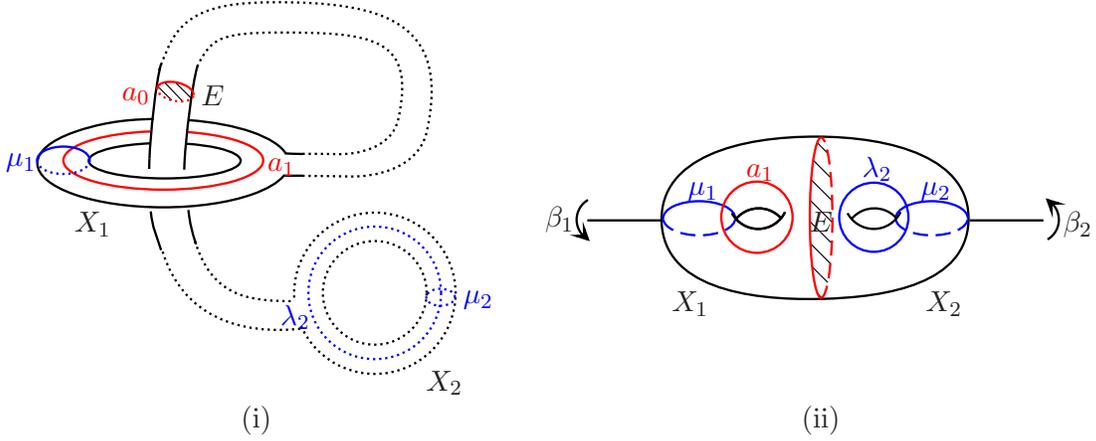}
  \linethickness{3pt}
  \put(83,0){(i)}
  \put(295,0){(ii)}
  \put(38,124){\color{red} $a_0$}
  \put(93,97){\color{red} $a_1$}
  \put(274,95){\color{red} $a_1$}
  \put(68,122){$E$}
  \put(298,74){$E$}
  \put(198,76){$\beta_1$}
  \put(394,74){$\beta_2$}
  \put(20,74){$X_1$}
  \put(153,14){$X_2$}
  \put(246,44){$X_1$}
  \put(342,44){$X_2$}
  \put(-6,99){\color{blue} $\mu_1$}
  \put(167,47){\color{blue} $\mu_2$}
  \put(98,39){\color{blue} $\lambda_2$}
  \put(252,88){\color{blue} $\mu_1$}
  \put(340,88){\color{blue} $\mu_2$}
  \put(318,95){\color{blue} $\lambda_2$}
\end{overpic}
\captionof{figure}{(i) The curves $a_0$, $a_1$, $\mu_1$, $\mu_2$ and $\lambda_2$ on $\partial V$; 
(ii) The maps $\beta_1$ and $\beta_2$.}
\label{fig:type2-2_claim}
\end{center}

\begin{claim}
\label{claim:Type 3-3}
The W-system $\mathcal{W} $ does not contain Type $3$-$3$ annuli.  
\end{claim}
\begin{proof}[Proof of Claim $\ref{claim:Type 3-3}$]
Suppose that $\mathcal{W}$ contains a Type 3-3 annulus $A$. 
Then there exist two disjoint solid tori $X_1$ and $X_2$ in $S^3$, and 
an  arc $\tau$ connecting $\partial X_1$ and $\partial X_2$ with 
$V = X_1 \cup X_2 \cup \Nbd (\tau)$ such that 
$A$ is properly embedded in 
$E(X_1 \sqcup X_2)$, and $a_i := A \cap X_i$ ($i=1,2$) is  
a simple closed curve on 
$\partial X_i$ that does not bound a disk in $X_i$. 
Let $E$ be the cocore of the 1-handle $\Nbd(\tau)$, see Figure \ref{fig:type3-3_claim} (i). 
\begin{center}
\begin{overpic}[width=10cm,clip]{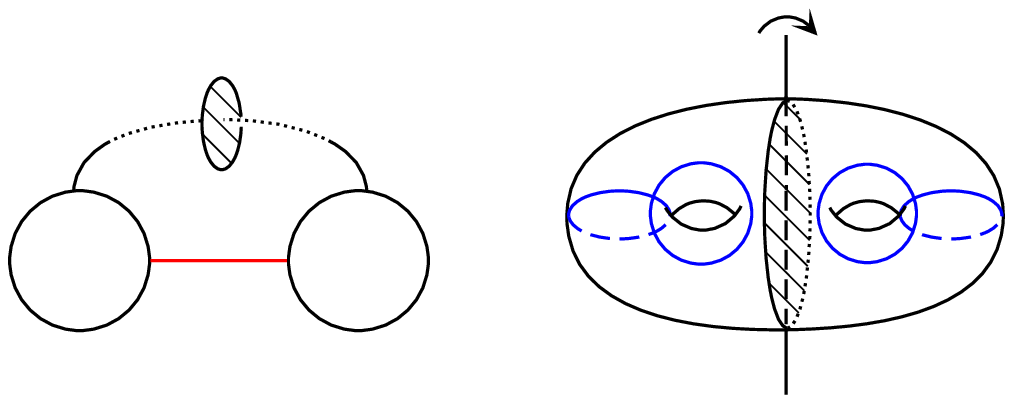}
  \linethickness{3pt}
  \put(57,0){(i)}
  \put(213,0){(ii)}
  \put(19,55){$X$}
  \put(96,55){$X$}
  \put(170,33){$X_1$}
  \put(260,33){$X_2$}
  \put(168,82){\color{blue} $\mu_1$}
  \put(260,82){\color{blue} $\mu_2$}
  \put(190,89){\color{blue} $\lambda_1$}
  \put(235,89){\color{blue} $\lambda_2$}
  \put(27,93){$\tau$}
  \put(216,132){$\gamma$}
  \put(59,114){${E}$}
  \put(214,67){${E}$}
  \put(58,62){\color{red} $A$}
\end{overpic}
\captionof{figure}{(i) The disk $E$ and the preferred longitudes $\lambda_i$ ($i=1,2$) 
of $X_i$; (ii) The map $\gamma$.}
\label{fig:type3-3_claim}
\end{center}
By Lemma \ref{case2}, $E$ is the unique separating disk in $V$ disjoint from 
$a_1 \cup a_2$. 
Let $\mu_i$ and $\lambda_i$ ($i=1,2$) be the meridian and 
the preferred longitude of $X_i$. 
By the uniqueness of $\mathcal{W}$ (Lemma \ref{W-system}), 
there exists a natural number $n$ with $\hat{f}^n (A) = A$. 
Using $2n$ instead of $n$ if necessary, 
we can assume that 
$\hat{f}^n|_V$ is orientation-preserving. 
Since $\hat{f}^n$ preserves $a_1 \cup a_2$, 
it preserves $E$ as well by Lemma \ref{case2}. 
Thus $\mu_1 \cup \mu_2$ and $\lambda_1 \cup \lambda_2$ are also preserved 
by $\hat{f}^n|_V$. 
It follows again from Cho \cite{Cho} that $\hat{f}^n |_V$ is 
a composition of $\alpha$, $\beta_1$, $\beta_2$ defined in Case 1, and one more element 
$\gamma$ shown in Figure \ref{fig:type3-3_claim} (ii). 
The hyperelliptic involution $\alpha$ is commutative with the other maps. 
The maps $\beta_1$, $\beta_2$ and $\gamma$ satisfy 
${\beta_1}^2|_{\partial V} = {\beta_2}^2|_{\partial V} =T_{\partial E}$, 
$\beta_1 \circ {\beta_2}^{-1} = \alpha$ and 
$\gamma \circ \beta_1 = \beta_2 \circ \gamma$. 
Hence, we can write $\hat{f}^n|_V$ in the form
\[
\hat{f}^n|_V = \alpha^{\epsilon_1} \circ {\beta_1}^{n_{1}} \circ {\beta_2}^{n_{2}} \circ \gamma ^{\epsilon_2} ~(\epsilon_i \in \{ 0,1\}, ~ n_i \in \Integer) . \] 
Then it holds 
\begin{eqnarray*}
(\hat{f}^n|_V)^4 = \left\{ 
\begin{array}{ll}
{\beta_1}^{4n_{1}} \circ {\beta_2}^{4n_{2}} & (\mbox{if }\epsilon _2 = 0); \\
{\beta_1}^{2(n_{1} + n_{2})} \circ {\beta_2}^{2(n_{1} + n_{2})} & (\mbox{if }\epsilon_2 = 1),\\
\end{array}
\right.
\end{eqnarray*}
hence $(\hat{f}^n|_{\partial V})^4 = {T_{\partial E}}^{2(n_{1} + n_{2})}$.  
The rest of the proof runs as in Claim \ref{claim:Type 2-2}. 
\end{proof}

\begin{claim}
\label{claim:Type 4-1}
The W-system $\mathcal{W} $ does not contain Type $4$-$1$ annuli.  
\end{claim}
\begin{proof}[Proof of Claim $\ref{claim:Type 4-1}$]
Suppose that $\mathcal{W}$ contains a Type 4-1 annulus $A$. 
The two circles $\partial A$ are parallel on $\partial V$, thus 
cobound an annulus $A^\prime$ on $\partial V$. 
We note that the union $A \cup A^\prime$ is the boundary of 
a regular neighborhood $\Nbd(K)$ of an 
Eudave-Mu\~noz knot $K$. 
The annulus $A$ cuts $M$ into the solid torus $\Nbd(K)$ and 
a handlebody $W$ of genus two. 
The core $k_{A'}$ of the annulus $A'$ is an essential simple closed curve on $\partial V$. 
Set $P := \Cl (\partial V \setminus A')$. 
$P$ is then a two-holed torus.

Since $P$ is incompressible in $E(K)$, $(V, k_{A'})$ is boundary-irreducible. 
Thus, by Lemma \ref{rel-handlebody}, 
$\MCG (V, k_{A'}) / \mathcal{A}(V, k_{A'})$ is a finite group. 
By Lemma \ref{Case3-lemm}, 
the relative handlebody $(V, k_{A^\prime})$ admits no non-separating essential annuli. 
Since the genus of $V$ is two, the boundary circles of any incompressible, separating annulus $B$ 
in $V$ is parallel on $\partial V$. 
Indeed, each boundary circle of $B$ is essential and separating on $\partial V$, and 
two disjoint, separating, essential, simple closed curves on $\partial V$ are parallel on $\partial V$. 
Hence the restriction $T_B|_{\partial V}$ of the twist $T_B$ along $B$ is trivial 
in $\MCG (\partial V)$. 
It then follows from the well-known fact that 
the natural homomorphism $\MCG (V) \to \MCG(\partial V)$ 
taking $h \in \MCG(V)$ to $h|_{\partial V} \in \MCG(\partial V)$ is injective, 
that $\MCG(V, k_{A'})$ it self is a finite group. 
By the uniqueness of $\mathcal{W}$ (Lemma \ref{W-system}), 
there exists a natural number $n$ with $\hat{f}^n (A) = A$. 
In particular, $\hat{f}^n |_V (k_{A'}) \in \MCG (V, k_{A'})$. 
Therefore, the order of $\hat{f}^n|_V$ (and so the order of $f$) is finite, 
which contradicts Lemma \ref{lem:finite order}.
\end{proof}

\begin{claim}
\label{claim:Type 2-1}
The W-system $\mathcal{W} $ does not contain Type $2$-$1$ annuli.  
\end{claim}
\begin{proof}[Proof of Claim $\ref{claim:Type 2-1}$]
Suppose that $\mathcal{W}$ contains a Type 2-1 annulus $A$. 
By Lemmas \ref{lem:uniqueness of Type 2-1}, \ref{lem:Types 2-1 and 3-2}, and 
Claims \ref{claim:Type 2-2}--\ref{claim:Type 4-1}, we have $\mathcal{W} = \{ A \}$. 

In the following, we show that the order of $f$ is finite, 
which contradicts Lemma \ref{lem:finite order}.
Let $M_0$ be the result of cutting $M$ off by the W-system $\mathcal{W} = \{A\}$. 
Then $M_0$ is also atoroidal.  
Indeed, suppose not, 
that is, there exists an essential torus $T$ in $M_0$. 
Since $M$ is atoroidal, $T$ admits a compressing disk $D$ in $M$ with $D \cap A \neq \emptyset$. 
We isotope $D$ so that $\# (D \cap A)$ is minimal.  
Note that $D \cap A$ is non-empty and consists of simple closed curves on $D$. 
Let $c$ be a component of $D \cap A$ that is innermost in $D$. 
Then the disk $D_0 \subset D$ bounded by $c$ is a compressing disk for $A$, 
which is a contradiction. 
Therefore $M_0$ is atoroidal. 
Now by considering the characteristic compression body of $M_0$, 
we see that $M_0$ is either boundary-irreducible or a handlebody of genus two 
(otherwise $M_0$ contains an essential torus, a contradiction). 
Set $\partial A = a_1 \sqcup a_2$, where 
$a_1$ bounds a disk in $V$, and 
$a_2$ is a primitive curve on $\partial V$. 

\vspace{1em}

\noindent
{\it Case $1$}: The case where $M_0$ is boundary-irreducible. 

By Lemma \ref{prop-simple}, $(M_0 , \partial_0 M_0)$ is simple. 
Thus, by Proposition 27.1 of Johannson \cite{Jo2}, 
$\MCG(M_0, \partial_0 M_0)$ is a finite group, 
where we recall that 
$\partial _0 M_0 := \partial M_0 \cap \partial M$ ($= \partial V \setminus \Nbd(a_1 \cup a_2)$). 
Therefore, we have 
$\hat{f}^n|_{\partial V \setminus (\Nbd (a_1 \cup a_2))} = \mathrm{id}$ 
for a natural number $n$. 
It follows that there exist natural numbers $m_1$ and $m_2$ satisfying 
\[ f^n = \hat{f}^n|_{\partial V} = {T_{a_1}}^{m_1} \circ {T_{a_2}}^{m_2} , \]
where $T_{a_i}$ ($i=1,2$) is the Dehn twist along $a_i$. 
By Theorem \ref{McC}, we have $m_1 = m_2 = 0$. 
Indeed, if $m_1 = 0$ and $m_2 \neq 0$, or if $m_1 \neq 0$ and $m_2 = 0$, 
$M$ is a boundary-reducible, which is a contradiction. 
If $m_1 \neq 0$ and $m_2 \neq 0$, 
either each of $a_1$ and $a_2$ bounds a disk in $V$, or 
$a_1$ and $a_2$ cobound an incompressible annulus in $V$. 
Since $a_2$ is primitive in $\partial V$, the former is impossible. 
The latter is also impossible because 
as an element of $H_1(V)$, $[a_1] = 0$ whereas $[a_2] \neq 0$. 
Now we have $f^n = \mathrm{id}$, that is, the order of $f$ is finite. 

\vspace{1em}

\noindent
{\it Case $2$}: The case where $M_0$ is a handlebody of genus two. 

The union $V_0 := V \cup \Nbd(A)$ is still a handlebody of genus two with $M_0 = E(V_0)$. 
That is, $S^3 = V_0 \cup M_0$ is a Heegaard splitting. 
Define $\hat{A} \subset \Nbd(A) = A \times I$ by 
$\hat{A} = ((S^1 \times \{ 1/2 \}) \times I)$ (see Figure \ref{fig:type2-1_claim}). 
\begin{center}
\begin{overpic}[width=10cm,clip]{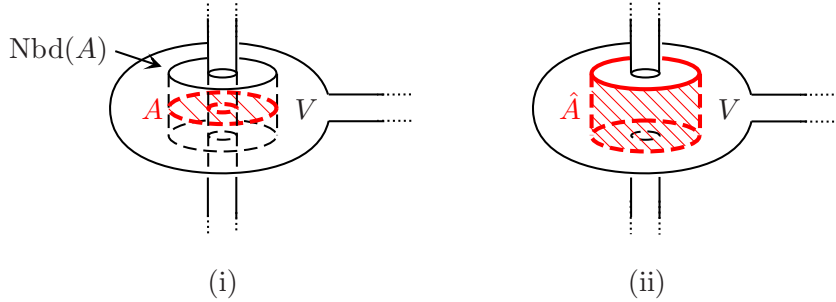}
  \linethickness{3pt}
  \put(42,0){(i)}
  \put(201,0){(ii)}
  \put(-33,88){$\Nbd(A)$}
  \put(75,65){$V$}
  \put(235,65){$V$}
  \put(17,65){\color{red} $A$}
  \put(175,65){\color{red} $\hat{A}$}
\end{overpic}
\captionof{figure}{(i) The annulus $A$; (ii) The annulus $\hat{A}$.}
\label{fig:type2-1_claim}
\end{center} 
Set $\partial \hat{A} = \hat{a}_1 \sqcup \hat{a}_2$. 
The simple closed curve $\hat{a}_i$ ($i=1,2$) bounds a non-separating disk 
$D_i$ in $V_0$. 
We note that $(V_0 \setminus \Nbd(\hat{A}) )\cong (V_0 \setminus \Nbd(A) )\cong V$.

Suppose that both $D_1$ and $D_2$ are primitive disks in $V_0$. 
Since $D_1$ and $D_2$ are disjoint, non-parallel, primitive 
disks in $V_0$, there exist 
disjoint, non-parallel, primitive 
disks $E_1$ and $E_2$ in $V_0$ such that 
$\# (D_i \cap E_j ) = \delta_{ij}$ (see Lemma 2.2 in Cho \cite{Cho}), 
where $\delta_{ij}$ is the Kronecker delta. 
Thus $V$ can be isotoped as shown in Figure \ref{fig:isotopy_of_v} 
(see also Theorem 4.4 of \cite{Koda}). 
It follows that $M = E(V)$ is a handlebody of genus two, which is a contradiction. 

Suppose that at least one of $D_1$ and $D_2$, say $D_1$, is not a primitive disk in $V_0$. 
Since $\hat{f}^2|_{V_0}$ preserves $D_1$, 
$\hat{f}^2$ is an automorphism of $(S^3, V_0, D_1)$. 
By Proposition 10.1 of Cho-McCullough \cite{Cho-McC}, 
we have $\hat{f}^4|_{V_0} =  \mathrm{id}$. 
This implies that $\hat{f}^4 |_ {\partial V \setminus (\Nbd(a_1 \cup a_2))}$ is the identity. 
Now, the finiteness of the order of $f$ follows from the 
same argument as in Case 1. 
\end{proof}
\begin{center}
\begin{overpic}[width=14.5cm,clip]{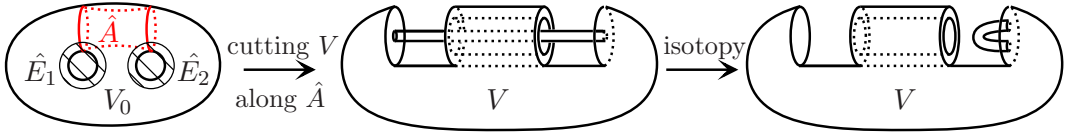}
  \linethickness{3pt}
  \put(41,39.5){\color{red} $\hat{A}$}
  \put(13,25){$\hat{E_1}$}
  \put(71,25){$\hat{E_2}$}
  \put(43,14){$V_0$}
  \put(188,14){$V$}
  \put(342,14){$V$}
  \put(90,35){\small cutting $V$}
  \put(93,15){\small along $\hat{A}$}
  \put(255,35){\small isotopy}
\end{overpic}
\captionof{figure}{The isotopy of $V$.}
\label{fig:isotopy_of_v}
\end{center} 

The following is our final claim. 

\begin{claim}
\label{claim:Type 3-2}
The W-system $\mathcal{W} $ does not contain Type $3$-$2$ annuli.  
\end{claim}
\begin{proof}[Proof of Claim $\ref{claim:Type 3-2}$]

Suppose first that $\mathcal{W}$ contains a Type 3-2 annulus $A$ 
satisfying the property $\mathrm{(ii)}$ in Lemma 
$\ref{lem:A is an essential annulus in E(X)}$. 
To get a contradiction in this case, we will show that the order of $f$ is finite. 
Set $\partial A = a_1 \sqcup a_2$. 
By Lemma \ref{lem:uniqueness of X for Type 3-2}, 
for $A$, the solid torus $X$ in the definition of a 
Type 3-2 annulus is uniquely determined. 
By the assumption, $A$ is parallel to an annulus $A'$ on $\partial X$. 
Let $Y$ be the parallelism region between $A$ and $A'$, which is the solid torus. 
Note that the simple closed curve $a_1$ (and so $a_2$) is primitive for $Y$. 
Let $\tau$ be an arc contained in $Y$, with endpoints on $A'$, such that 
$V = X \cup \Nbd (\tau)$. 
The annulus $A$ separates $M$ into two components $M_0$ and $M_1$, where 
$M_0 \subset Y$. 
By the same argument as in Claim \ref{claim:Type 2-1}, we can check that both 
$M_0$ and $M_1$ are atoroidal. 
Let $E$ be the cocore of the 1-handle $\Nbd(\tau)$. 
By the uniqueness of $\mathcal{W}$ (Lemma \ref{W-system}), 
it holds $\hat{f}^{n_1} (A) = A$, $f^{n_1} (a_1) = a_1$ and $\hat{f}^{n_1}|_A$ is orientation 
preserving for some natural number $n_1$.  
By Lemma \ref{2.2}, $E$ is the unique non-separating disk in $V$ disjoint from $a_1$. 
Thus we have $f^{n_1} (\partial E) = \partial E$. 

\begin{subclaim}
\label{subclaim:claim 7 parallel case}
The image of the natural map 
$\MCG(M_0, a_1, \partial E) \to \MCG(\partial M_0, a_1, \partial E)$ 
that takes $\varphi \in \MCG(M_0, a_1, \partial E)$ to 
$\varphi|_{\partial M_0}$ is a finite group. 
\end{subclaim}
\begin{proof}[Proof of Subclaim]
Suppose that there exists an incompressible annulus $B$ in $M_0$ whose boundary is disjoint from 
$a_1 \cup \partial E$. 
Remark that $B$ is separating in $Y$, thus $B$ is separating in $M_0$ as well. 
Set $\partial B = b_1 \sqcup b_2$. 
Since $B$ is disjoint from $\partial E$ we may assume that $b_1 \sqcup b_2$ lies in $\partial Y$. 
We will show that $b_1$ and $b_2$ are parallel on $\partial M_0$. 

Assume that $b_1$ and $b_2$ are not parallel on $\partial M_0$ for a contradiction. 
Since $\partial B$ is disjoint from $a_1 \subset \partial Y$, 
$b_1$ and $b_2$ are both essential or both inessential on $\partial Y$. 
The attaching region of the 1-handle $\Nbd(\tau)$ consists of two disks 
$D_+$ and $D_-$ on $\partial X \cup \partial Y$. 
Both $D_+$ and $D_-$ are parallel to $E$. 
If both $b_0$ and $b_1$ are essential on $\partial Y$, they are parallel on $\partial Y$. 
Thus in this case there exists an annulus $F$ on $\partial Y$ with $\partial F = b_1 \sqcup b_2$. 
Note that the closure $F'$ of $\partial Y \setminus F$ is also an annulus, and 
$F$ and $F'$ are parallel in $Y$. 
Since $b_1$ and $b_2$ are not parallel on $\partial M_0$ by the assumption, 
we have $F \cap (D_+ \cup D_-) = D_+$ after changing the names of $F$ and $F'$, 
and $D_+$ and $D_-$ if necessary. See Figure \ref{fig:annulus_B} (i). 
\begin{center}
\begin{overpic}[width=12cm,clip]{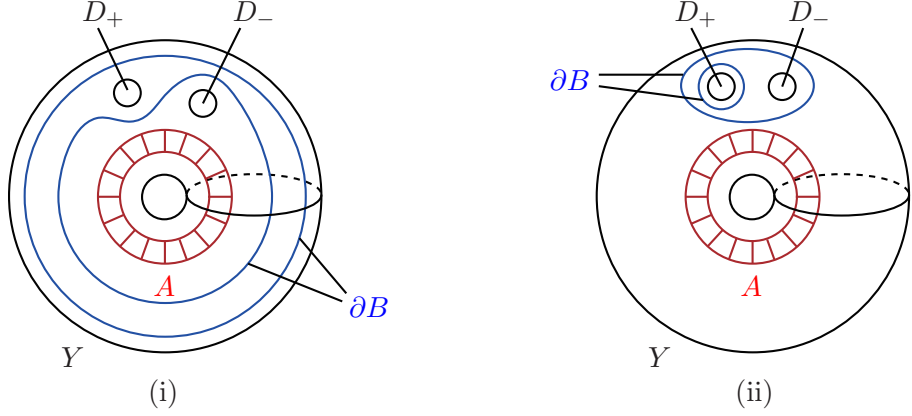}
  \linethickness{3pt}
  \put(53,0){(i)}
  \put(275,0){(ii)}
  \put(20,13){$Y$}
  \put(242,13){$Y$}
  \put(28,144){$D_+$}
  \put(252,144){$D_+$}
  \put(85,144){$D_-$}
  \put(295,144){$D_-$}
  \put(55,40){\color{red} $A$}
  \put(277,40){\color{red} $A$}
  \put(129,32){\color{blue} $\partial B$}
  \put(205,118){\color{blue} $\partial B$}

\end{overpic}
\captionof{figure}{The boundary loops of $B$ on $\partial Y$.}
\label{fig:annulus_B}
\end{center}
This implies that $B$ is non-separating in $M_0$, a contradiction. 
Thus both $b_1$ and $b_2$ are inessential on $\partial Y$. 
Each simple closed curve $b_i$ ($i=1,2$) bounds a disk $D_i$ on $\partial Y$. 
Since $B$ is incompressible in $M_0$, and $b_1$ and $b_2$ are not parallel on $\partial M_0$, 
we have $D_+ \subset D_1 \subset D_2$, $D_- \not\subset D_1$ and $D_- \subset D_2$ 
after changing the names of $b_1$ and $b_2$, 
and $D_+$ and $D_-$ if necessary. See Figure \ref{fig:annulus_B} (ii). 
This implies that $B$ is non-separating in $M_0$, again a contradiction. 
In consequence, $b_1$ and $b_2$ are parallel on $\partial M_0$. 

Now by exactly the same reason 
as Lemma \ref{rel-handlebody} (see Johannson \cite{Jo2}), the order of the group 
$\MCG (M_0, a_1, \partial E) / \mathcal{A} (M_0, a_1, \partial E)$ is finite, 
where $\mathcal{A} (M_0, a_1, \partial E)$ is the subgroup of the mapping class
group $\MCG (M_0, a_1, \partial E)$ generated by all twists along incompressible annuli in $M_0$ 
disjoint from $a_1 \cup \partial E$. 
Hence the restriction $T_B|_{\partial M_0}$ of the twist $T_B$ along any incompressible annulus $B$ in $M_0$ 
disjoint from $a_1 \cup \partial E$ is trivial in $\MCG (\partial M_0)$. 
This implies the assertion. 
\end{proof}
By the above Subclaim, 
there exists a natural number $n_2$ such that 
$\hat{f}^{n_1 n_2} |_{\partial M_0}$ is the identity as an element of $\MCG(\partial M_0, a_1, \partial E)$. 
Since $\MCG_+(A)$ is the trivial group (while $\MCG_+(A ~\rel~ \partial A) = \Integer$), 
we may assume that $\hat{f}^{n_1 n_2} |_{\partial M_0 \cap \partial V}$ is trivial as an element of 
$\MCG (\partial M_0 \cap \partial V)$. 
Therefore, we have 
$f^{n_1 n_2} = T_c^m$, where 
$c$ is the core of the annulus $\Cl ( \partial X \setminus A )$, 
$T_c$ is the Dehn twist along $c$, and 
$m$ is an integer. 
If $m \neq 0$, the simple closed curve $c$ bounds a disk $D$ in $M$ by Lemma \ref{McC}, a contradiction. 
Thus we have $m=0$, which implies that the order of $f$ is finite, which is a contradiction. 

\vspace{1em}

Now the only remaining possibility is that $\mathcal{W}$ consists of only Type 3-2 annuli  
satisfying the property $\mathrm{(i)}$ in Lemma 
$\ref{lem:A is an essential annulus in E(X)}$. 
Let $A \in \mathcal{W}$.  
By Lemma \ref{lem:uniqueness of X for Type 3-2}, 
for $A$, the solid torus $X$ in the definition of a 
Type 3-2 annulus is uniquely determined. 
Further, $X$ is a neighborhood of a torus knot or a cable knot, 
and $A$ is the cabling annulus for $X$. 
The number $b$ of isotopy classes in $\partial V$ of 
the boundary circles of the annuli in $\mathcal{W}$ 
is at most three. 

Suppose first that $b = 2$. 
Then there exists a Type 3-2 annulus $A'$ in $\mathcal{W}$ whose boundary circles are 
not parallel to those of $A$. 
By the uniqueness of $\mathcal{W}$ (Lemma \ref{W-system}), 
it holds $\hat{f} |_V (\partial A \cup \partial A' ) = \partial A \cup \partial A'$. 
Set $\partial A = a_1 \sqcup a_2$ and $\partial A' = a_1' \sqcup a_2'$.  
Since $a_1$ and $a_1'$ are non-separating and not parallel on $\partial V$, 
attaching 2-handles to $V$ along $a_1$ and $a_1'$, and then capping off 
the resulting boundary by a 3-handle, we obtain a closed 3-manifold $N$. 
Suppose first that $N$ is homeomorphic to $S^3$. 
By applying Alexander's tricks step by step, $\hat{f}|_V$ extends to an automorphism of $N$. 
This implies that $f$ is standardly extendable over $S^3$, which is a contradiction. 
Thus $N$ is not homeomorphic to $S^3$ (so $N$ is a lens space or $\Real P^3$); we do not know if this is the case. 
In this case, since both $a_1$ and $a_1'$ are primitive with respect to $V$ and 
$N \not\cong S^3$,  
the relative handleobody $(V, k)$ is boundary-irreducible, 
where $k = a_1 \cup a_1'$. 
By Lemma \ref{rel-handlebody}, $\MCG (V, k) / \mathcal{A} (V, k)$ is a finite group. 
We can further show that $(V, k)$ admits no non-separating essential annuli. 
Indeed, if there exists a non-separating annulus in $(V, k)$, 
it remains to be non-separating in the 3-manifold obtained by attaching a 2-handle to $V$ along $a_1$, 
which is a solid torus. 
This is a contradiction. 
Now, the same argument as in Claim \ref{claim:Type 4-1} shows that the order of $f$ is finite, 
which contradicts Lemma \ref{lem:finite order}.
 
Similar arguments apply to the case of $b=3$. 

In what follows, we suppose that $b=1$. 
Then $\mathcal{W}$ consists of one or two annuli, and in the latter case, 
the two annuli are parallel in $E(X)$ and the simple arc $\tau$ in the definition of a 
Type 3-2 annulus is contained in the parallelism region. 
See Figure \ref{fig:type3-2_claim}. 
\begin{center}
\begin{overpic}[width=11cm,clip]{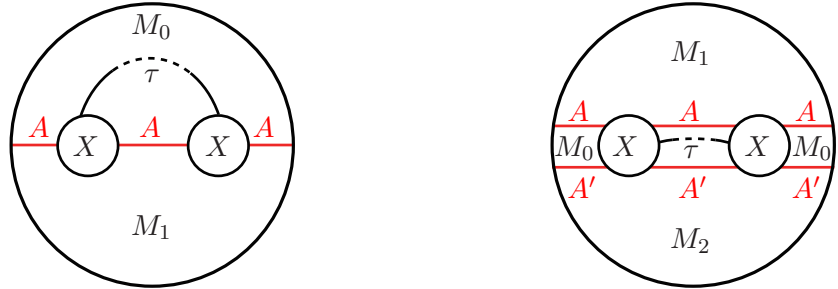}
  \linethickness{3pt}
  \put(24,50){$X$}
  \put(74,50){$X$}
  \put(46,96){$M_0$}
  \put(46,20){$M_1$}
  \put(51,78){$\tau$}
  \put(7,57){\color{red} $A$}
  \put(49,57){\color{red} $A$}
  \put(92,57){\color{red} $A$}

  \put(228,50){$X$}
  \put(278,50){$X$}
  \put(206,50){$M_0$}
  \put(296,50){$M_0$}
  \put(250,86){$M_1$}
  \put(250,15){$M_2$}
  \put(255,49){$\tau$}
  \put(211,63){\color{red} $A$}
  \put(253,63){\color{red} $A$}
  \put(297,63){\color{red} $A$}
  \put(211,34){\color{red} $A'$}
  \put(253,34){\color{red} $A'$}
  \put(296,34){\color{red} $A'$}
\end{overpic}
\captionof{figure}{The W-decompositions for $E(V)$.}
\label{fig:type3-2_claim}
\end{center} 

First, let $\mathcal{W} = \{ A, A' \}$. 
Let $M_0$, $M_1$, $M_2$ be the result of cutting $M$ off by the W-system $\mathcal{W} = \{A, A'\}$, 
where $M_0 \cap \Nbd(\tau) \neq \emptyset$ for the simple arc $\tau$ in $E(X)$ with 
$V = X \cup \Nbd(\tau)$ (see the right-hand side in Figure \ref{fig:type3-2_claim}). 
Both of $M_1$ and $M_2$ are solid tori, otherwise $E(V)$ is toroidal and it contradicts our assumption. 
We note that $\partial M_0$ is a closed surface of genus two.  
By the same argument as in Claim \ref{claim:Type 2-1}, $M_0$ is atoroidal. 
Thus by considering the characteristic compression body of $M_0$, 
we see that $M_0$ is either boundary-irreducible or a handlebody of genus two. 

Suppose that $M_0$ is boundary-irreducible. 
By Lemma \ref{prop-simple}, $(M_0 , \partial_0 M_0)$ is simple. 
Due to Proposition 27.1 of Johannson \cite{Jo2}, 
$\MCG(M_0, \partial_0 M_0)$ is a finite group. 
Thus, there exists a natural number $n$ with 
$\hat{f}^n|_{{\rm Cl}(\partial V \setminus (\partial M_1 \cup \partial M_2))} = 
\mathrm{id}_{{\rm Cl}(\partial V \setminus (\partial M_1 \cup \partial M_2))}$. 
Let $c_i$ ($i=1,2$) be the core of the annulus $\partial V \cap \partial M_i $ 
$(= \partial X \cap \partial M_i)$. 
Then for some $m_1$, $m_2$ we can write $f^n$ in the form 
$f^n = \hat{f}^n|_{\partial V} = {T_{c_1}} ^{m_1} \circ {T_{c_2}} ^{m_2}$. 
If $m_i \neq 0$, this cannot be extended to the solid torus $M_i$. 
Hence $m_1 = m_2 = 0$. 
Therefore the order of $f$ is finite, which is a contradiction. 

Suppose that $M_0$ is a handlebody of genus two. 
Let $Y$ be the parallelism region of $A$ and $A'$ in $E(X)$. 
Since $Y$ is a solid torus, the arc $\tau$ is trivial in $Y$ by Gordon \cite{Gor87}. 
Let $l$ be the boundary circle of the cocore of the 1-handle $\Nbd (\tau)$. 
Note that $l$ is primitive with respect to the handlebody $M_0$. 
Set $\partial A = a_1 \sqcup a_2$, $\partial A' = a_1' \sqcup a_2'$ and $k = a_1 \cup a_1' \cup l$. 
Since $E(V)$ is boundary-irreducible, the relative handlebody $(M_0, k)$ is boundary-irreducible. 
If there exists a non-separating annulus in $(V, k)$, 
it remains to be non-separating in the 3-manifold obtained by attaching a 2-handle to $V$ along $l$, 
which is a solid torus. 
This is a contradiction. 
Now, the same argument as in Claim \ref{claim:Type 4-1} shows that the order of $f$ is finite, 
which contradicts Lemma \ref{lem:finite order}.


Finally, let $\mathcal{W} = \{ A \}$.
Let $M_0$, $M_1$ be the result of cutting $M$ off by the W-system $\mathcal{W} = \{A\}$, 
where $M_0 \cap \Nbd(\tau) \neq \emptyset$ for the simple arc $\tau$ in $E(X)$ with 
$V = X \cup \Nbd(\tau)$ (see the left-hand side in Figure \ref{fig:type3-2_claim}). 
We note that $\partial M_0$ is a closed surface of genus two.  
By the same argument as in Claim \ref{claim:Type 2-1}, $M_0$ is atoroidal. 
Thus by considering the characteristic compression body of $M_0$, 
we see that $M_0$ is either boundary-irreducible or a handlebody of genus two. 

Suppose that $M_0$ is boundary-irreducible. 
By Lemma \ref{prop-simple}, $(M_0 , \partial_0 M_0)$ is simple. 
Due to Proposition 27.1 of Johannson \cite{Jo2}, 
$\MCG(M_0, \partial_0 M_0)$ is a finite group. 
Thus, there exists a natural number $n$ with 
$\hat{f}^n|_{{\rm Cl}(\partial V \setminus \partial M_1)} = \mathrm{id}_{{\rm Cl}(\partial V \setminus \partial M_1)}$. 
Let $c$ be the core of the annulus $\partial V \cap \partial M_1 $ 
$(= \partial X \cap \partial M_1)$. 
Then for some $m$ we can write $f^n$ in the form 
$f^n = \hat{f}^n|_{\partial V} = {T_c} ^m$. 
If $m \neq 0$, then $c$ bounds a disk in $M$ by Theorem \ref{McC}. 
This contradicts the boundary-irreducibility of $M$. 
Hence $m = 0$. 
Therefore the order of $f$ is finite, which is a contradiction. 

Suppose that $M_0$ is a handlebody of genus two. 
Since $A$ is the cabling annulus for $X$, and the annulus $A$ separates $E(X)$ 
into $M_0 \cup \Nbd(\tau)$ and $M_1$, 
at least one of $M_0 \cup \Nbd(\tau)$ and $M_1$ is a solid torus. 

\vspace{1em}

\noindent
{\it Case $1$}: The case where $M_0 \cup \Nbd(\tau)$ is a solid torus and 
$M_1$ is the exterior of a non-trivial knot. 

This is impossible because, in this case, 
$ T := \partial (X \cup \Nbd(A)) \cap M_1$ is an essential torus in $M$ 
(see Figure \ref{fig:v_0} (i)), 
which contradicts the 
assumption that $M$ is atoroidal. 
\begin{center}
\begin{overpic}[width=11cm,clip]{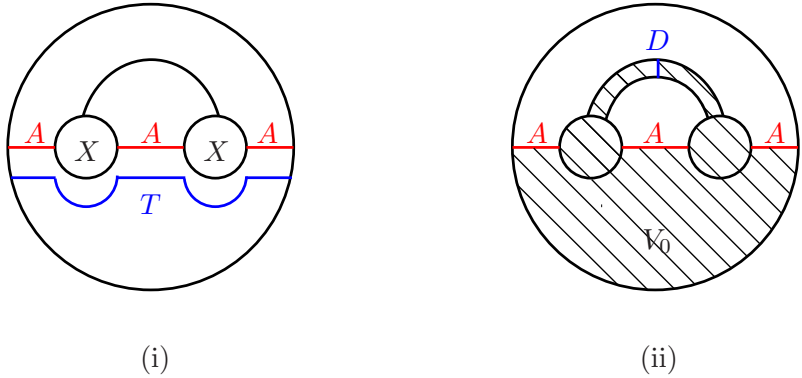}
  \linethickness{3pt}
  \put(57,0){(i)}
  \put(246,0){(ii)}
  \put(57,58){\color{blue} $T$}
  \put(247,45){$V_0$}
  \put(32,78){$X$}
  \put(81,78){$X$}
  \put(13,85){\color{red} $A$}
  \put(57,85){\color{red} $A$}
  \put(101,85){\color{red} $A$}
  \put(203,85){\color{red} $A$}
  \put(247,85){\color{red} $A$}
  \put(293,85){\color{red} $A$}
  \put(248,120){\color{blue} $D$}
\end{overpic}
\captionof{figure}{(i) The essential torus $T$; 
(ii) The handleobody $V_0$ of genus two.}
\label{fig:v_0}
\end{center} 

\vspace{1em}

\noindent
{\it Case $2$}: The case where $M_0 \cup \Nbd(\tau)$ is 
the exterior of a non-trivial knot and 
$M_1$ is a solid torus. 

Set $V_0 := M_1 \cup X \cup \Nbd(\tau)$ (see Figure \ref{fig:v_0} (ii)). 
Let $D$ be the cocore of the 1-handle $\Nbd(\tau)$. 
We note that now $S^3 = V_0 \cup M_0$ is a Heegaard splitting of genus two, and 
$D$ is a non-separating disk in $V_0$.
The solid torus $\Cl ( V_0 \setminus \Nbd(D))$ 
($= E(M_0 \cup \Nbd(A))$ is a regular neighborhood of a non-trivial knot in $S^3$. 
The disk $D$ is thus not a primitive disk in $V_0$ due to Lemma 2.1 of Cho \cite{Cho}. 
By the uniqueness of $\mathcal{W}$, $X$, and $A$, 
$\hat{f}$ preserves $V$, $M_1$ and $D$. 
Thus, we can think of $\hat{f}$ as an automorphism of $ (S^3, V_0, D)$. 
By Proposition 10.1 of Cho-McCullough \cite{Cho-McC}, 
$\MCG (S^3, V_0, D)$ is a finite group. 
This implies that for some $n$, $\hat{f}^n$ is isotopic to the identity 
via an isotopy preserving $V_0$ and $D$. 
Set $B := X \cap M_1$. $B$ is an incompressible annulus in $V_0$. 
Since $\hat{f}$ preserves $V$, $f$ preserves $B$ as well. 
Then clearly the above isotopy from $\hat{f}^n$ to the identity 
also preserves $B$.   
Consequently, $\hat{f}^n$ is isotopic to the identity on $\partial V$ and 
so the order of $f$ is finite, which is a contradiction. 

\vspace{1em}

\noindent
{\it Case $3$}: The case where both $M_0 \cup \Nbd(\tau)$ and $M_1$ 
are solid tori. 

Set $V_0 := M_1 \cup X \cup \Nbd(\tau)$ (see again Figure \ref{fig:v_0} (ii)). 
Let $D$ be the cocore of the 1-handle $\Nbd(\tau)$. 
The solid torus $\Cl ( V_0 \setminus \Nbd(D))$ 
($\cong E(M_0 \cup \Nbd(A))$ is now a regular neighborhood of the trivial knot in $S^3$. 
The disk $D$ is then a primitive disk in $V_0$ (see Lemma 2.1 of Cho \cite{Cho}).
We can again think of $\hat{f}$ as an automorphism of $ (S^3, V_0, D)$. 
By the argument in Lemma 5.5 of \cite{Cho-Koda}, 
$\hat{f}^2|_{V_0}$ is a composition of $\alpha$, $\beta$, $\gamma$ shown in Figure \ref{fig:beta_gamma}. 
\begin{center}
\begin{overpic}[width=11cm,clip]{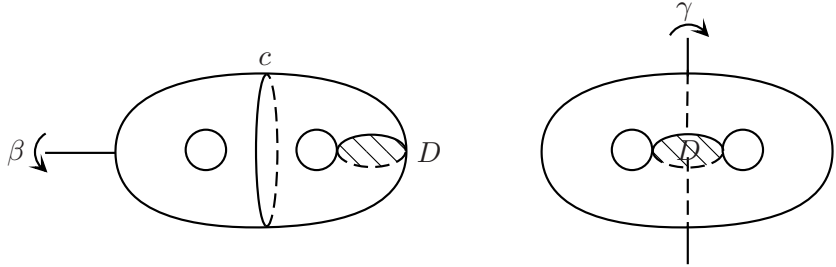}
  \linethickness{3pt}
  \put(-5,45){$\beta$}
  \put(248,100){$\gamma$}
  \put(90,81){$c$}
  \put(150,44){$D$}
  \put(248,45){$D$}
\end{overpic}
\captionof{figure}{The maps $\beta$ and $\gamma$.}
\label{fig:beta_gamma}
\end{center} 
Here $\alpha$ comes from the hyperelliptic involution of $\partial V$, 
$\beta$ is the half-twist of a handle, and $\gamma$ exchanges handles. 
Both $\alpha$ and $\gamma$ are order two elements, while 
the order of $\beta$ is infinite. 
We note that $\beta^2|_{\partial V_0}$ is the Dehn twist along 
a separating simple closed curve $c$ on $\partial V_0$. 
These elements satisfy 
$\alpha \circ \beta = \beta \circ \alpha$, $\alpha \circ \gamma = \gamma \circ \alpha$, 
and $\beta \circ \gamma =  \alpha \circ \gamma \circ \beta$.  
Therefore we can write $\hat{f}^2|_{V_0}$ in the form
\[
\hat{f}^2|_{V_0} = \alpha^{\epsilon_1}   \circ \gamma^{\epsilon_2} \circ \beta^n~
 (\epsilon_1, \epsilon_2 \in \{ 0,1\}, ~ n \in \Integer).  
\]
Thus we have $(\hat{f}^2|_{\partial V_0})^4 = {T_c}^{2n}$. 
Since $c \cap \partial A \neq \emptyset$, 
$n$ should be $0$. 
Consequently, for some $n$, $\hat{f}^n$ is isotopic to the identity 
via an isotopy preserving $V_0$ and $D$. 
The rest of the proof runs as in Case 2. 
\end{proof}

By Claims \ref{claim:Type 2-2}--\ref{claim:Type 3-2}, $M$ can contain no essential annuli, which is a contradiction. 

Consequently, $f$ is standardly extendable over $S^3$. 
This completes the proof of Theorem \ref{theo-main}. 

\section*{Acknowledgments} 
The authors wish to express their gratitude to 
Shicheng Wang and for helpful comments. 
They are also profoundly grateful to Mario Eudave-Mu\~noz for pointing out an error in the 
original draft.


\end{document}